\documentclass[a4paper,12pt]{article}
\usepackage{geometry}
\usepackage{cuted}
\usepackage{lipsum}
\usepackage{graphicx}
\usepackage{authblk}
\usepackage{chessfss}
\usepackage{xkeyval}
\usepackage{xifthen}
\usepackage{pgfcore}
\usepgfmodule{shapes}
\usepackage{pst-node}
\usepackage{chessboard}
\usepackage{float}
\usepackage[english]{babel}
\usepackage{csquotes}
\usepackage{tikz}
\usepackage[vlined,ruled]{algorithm2e}
\usepackage{xcolor}
\usepackage{algorithm2e}
\usepackage[title]{appendix}
\usepackage{ amssymb }
\usepackage[bottom]{footmisc}
\usepackage{cite}

\usepackage{graphicx}
\usepackage{lineno}
\usepackage{amsmath,amsfonts,amsthm}
\usepackage{amsthm}
\usepackage{chngcntr}

\newtheorem{theorem}{Theorem}

\theoremstyle{definition}
\newtheorem{definition}{Definition}[section]

\usepackage{tcolorbox}
\tcbuselibrary{theorems}

\theoremstyle{example}
\newtheorem{example}{Example}[section]

\newtcbtheorem{mydefinition}{Definition}%
{colback=green!8,colframe=green!35!black,fonttitle=\bfseries}{th}

\newtcbtheorem{mytheo}{Theorem}%
{colback=blue!8,colframe=blue!35!black,fonttitle=\bfseries}{th}

\usepackage[numbers,sort,compress]{natbib}
\usepackage[normal,footnotesize,labelfont=bf,up,textfont=it,up]{caption}	
\usepackage{booktabs}												
\usepackage{fix-cm}													
\usepackage{todonotes}
\usepackage[section]{placeins}
\let\Oldsection\section
\renewcommand{\section}{\FloatBarrier\Oldsection}

\let\Oldsubsection\subsection
\renewcommand{\subsection}{\FloatBarrier\Oldsubsection}

\let\Oldsubsubsection\subsubsection
\renewcommand{\subsubsection}{\FloatBarrier\Oldsubsubsection}

\usepackage{mathtools}

\usepackage{color}   
\usepackage{hyperref}
\hypersetup{
    colorlinks=true, 
    linktoc=all,     
    linkcolor=blue,  
}

\hypersetup{
    colorlinks=true,
    linkcolor=blue,
    filecolor=magenta,      
    urlcolor=blue,
}
\title{
Finite Atomized Semilattices }
\author[1,2]{Fernando Martin-Maroto}
\author[1,2]{Gonzalo G. de Polavieja}
\affil[1]{\small Champalimaud Research, Champalimaud Centre for the Unknown, Lisbon, Portugal} \affil[2]{\small Algebraic AI, Madrid, Spain }

\begin{document}
\maketitle
\begin{abstract}
We show that every finite semilattice can be represented as an atomized semilattice, an algebraic structure with additional elements (atoms) that extend the semilattice's partial order. Each atom maps to one subdirectly irreducible component, and the set of atoms forms a hypergraph that fully defines the semilattice. An atomization always exists and is unique up to ``redundant atoms''. Atomized semilattices are representations that can be used as computational tools for building semilattice models from sentences, as well as building its subalgebras and products. Atomized semilattices can be applied to machine learning and to the study of semantic embeddings into algebras with idempotent operators.   
\end{abstract}
\newpage

\tableofcontents

\newpage

\section{Introduction} \label{intro}

Elements in a semilattices, just like in Boolean algebras, can be represented as sets \cite{Stone}. We leverage this fact to extend the semilattice to an algebra of two sorts, the atomized semilattice, which is well suited as a computational tool to introduce additional properties in models as well as to carry out selected properties from model to model.
 
Our main results are the concept of atomization, redundancy and the full crossing construction which allows building semilattices from a set of positive atomic sentences. Special attention is paid to relative freedom between models as it may play a role in the field of machine learning. We provide a few examples to illustrate the power of the atomized semilattice and explore its most basic properties.  

Although this manuscript does not explicitly discuss machine learning, it is meant as a reference for researchers interested in the mathematics of algebraic machine learning \cite{Maroto}. We deal here with the purely algebraic aspects of the atomised semilattice construction, which may also be of interest to algebraists.

In section \ref{Description} we give an axiomatization for atomized semilattices, introduce the concepts of redundant atom and discriminant atom and also extend the semilattice homomorphism to atomized semilattices. We discuss some representational choices that endow the elements of an atomized semilattice, including atoms, with the "universal" property that allows to identify them across every model generated by the same constants.

In section \ref{freedomFullCrossingAndRedundacy} we start by introducing the notion of relative algebraic freedom (definition \ref{freerModelDefinition}). We prove that redundant atoms can be discarded (theorem \ref{redundantAtom}) and that finite semilattices can be represented with a unique atomization without redundant atoms (theorem \ref{uniqueAtomization}). We continue with the central concept of full crossing that permits building the freest model of a set of atomic sentences (theorem \ref{fullCrossingIsFreestTheorem}). 

Section \ref{unionJoinRestrictionSubalgebraAndProducts} is a selection of results. We show that some subalgebras can be constructed directly from the restriction of the atoms (theorem \ref{restrictionLemma}), and later show how to calculate any subalgebra in a simple way (theorem \ref{subalgebraTheorem}). In theorem \ref{construibleSubmodelTheorem} we extend the link (introduced in theorem \ref{unionWithFreer}) between freedom and redundancy. We explain how to build congruences using full crossing (theorem \ref{congruenceFromFullCrossing}) and use the congruences to build join models, subalgebras and products. We provide a method to calculate join models, which are the result of gluing together various semilattice models (see theorem \ref{sumOfModelsWithIntersection}), and discuss when it is possible to find embeddings from two semilattices into their join model (theorem \ref{embeddingInJoinModelTheorem}). Products of models are considered in theorem \ref{productModelNoIntersectTheorem}. Finally, the connection between the atoms and the irreducible components of a subdirect product (theorem \ref{modelAsASubdirectProduct}) is made explicit.

\section{Preliminaries} \label{Description}

A semilattice \cite{Burris} is an algebra with a single binary operation $\odot$ that is commutative, associative and idempotent. The idempotent operator induces a relationship of (partial) order. We say $a \leq b$ in a semilattice $M$ if and only if $M \models (b = b \odot a)$. The elements of a semilattice form a partially ordered set. We often refer to semilattices as semilattice models or, simply, models. 

A semilattice model ``over a set of constants $C$'' is assumed to be generated by its constants, i.e. every element can be found as an idempotent summation of constants. We only deal here with finite models so we assume $C$ is a finite set.  In order to get a model generated by its constants, if an element $b$ of a semilattice model is not equal to any idempotent summation of constants then it is possible to make $b$ a constant. Without loss of generality we assume in this text that every finite semilattice is generated by its constants.

Given a set $C$ of constants, the free algebra \cite{Burris} $F_C(\emptyset)$ contains a different element for each constant in $C$ and for each idempotent summation of constants in $C$, where two summations are considered equal if and only if they can be proved equal using the commutative, associative and idempotent properties of $\odot$. From these three properties follows easily that the cardinality of the freest model (the free algebra) is $2^{\vert C\vert} - 1$.  $F_C(\emptyset)$ is the well-know term algebra. The elements of $F_C(\emptyset)$ are called ``terms''. 

There is always a natural homomorphism \cite{Burris} $\nu_M$ from the term algebra $F_C(\emptyset)$ onto any model $M$ that maps terms to elements of $M$. The inverse image of the natural homomorphism $\nu_M : F_C(\emptyset) \rightarrow M$ partitions the terms over $C$ into as many equivalence classes as elements are in $M$.  We can see the terms in each equivalence class of $F_C(\emptyset)$ as an alternative name for the same element of $M$. Each semilattice model over $C$ is a partition of the set $F_C(\emptyset)$. Since the elements of a semilattice are equivalence classes defined in $F_C(\emptyset)$, every semilattice over $C$ have fewer than $2^{\vert C\vert} - 1$ elements. 

A potential source of confusion is that we do not give names to elements of models; instead, we give names to terms, i.e. we give names to elements of $F_C(\emptyset)$, names such as $t$ or $s$ for which it may happen in some model $M$ that $M \models (\nu_M(t) = \nu_M(s))$ or, simply, $M \models (t = s)$ as we do not explicitly write the homomorphism either. The reason for this unusual choice is that by using terms we can refer to the ``same element'' in different models. Atomized semilattices were introduced in the context of machine learning, where we do not stay working with the same model for long, instead, when new data challenges the model, the model gets replaced by another one. Using names for elements of particular models would force us to build explicit mappings when models are replaced and that turns out to be an unnecessary complication. 

A ``duple'' $r=(r_L,r_R)$ is defined in this text as an ordered pair of elements of $F_C(\emptyset)$ (terms). Positive and negative duples, written $r^{+}$ and $r^{-}$, are used as follows: we say $M \models r^{+}$ if and only if $M \models  (\nu_M(r_L) \leq \nu_M(r_R))$  and we say $M \models r^{-}$ if and only if $M \not\models (\nu_M(r_L) \leq \nu_M(r_R))$ or, without writing natural map symbol $M \models  (r_L \leq r_R)$. 

A semilattice model $M$ is atomized when it is extended with a set of additional elements we call ``atoms''. Each semilattice is a partial order and an atomization of the semilattice is an extension of this partial order. However, an atomized semilattice is not a semilattice extension; in atomized semilattices the idempotent operator is defined exclusively for regular elements while the order relation is defined for all, atoms and regular elements. If $a$ and $b$ are regular elements and $\phi$ is an atom, we can say $a \odot b$ and $\phi < a$ but not $a \odot \phi$. The regular elements of an atomized semilattice form a semilattice. All the elements of an atomized semilattice form a partially ordered set.
\bigskip
\bigskip

\begin{definition} \label{finiteAomizedSemilattice}   An atomized semilattice over a set of constants $C$ is a structure $M$ with elements of two sorts, the regular elements $\{a, b, c,...\}$ and the atoms $\{\phi, \psi,...\}$, with an idempotent, commutative and associative binary operator $\odot$ defined for regular elements and a partial order relation $<$ (i.e. a binary, reflexive, antisymmetric and transitive relation) that is defined for all the elements, such that the regular elements are either constants or idempotent summations of constants, and $M$ satisfies the axioms of the operations and the additional:
\[
\forall \phi \exists c : (c \in C) \wedge (\phi < c), \tag{AS1} 
\]\[
\forall \phi \forall a \,(a \not\leq \phi),  \tag{AS2} 
\]\[
\forall a \forall b \, (  a \leq b \,  \Leftrightarrow  \neg \exists \phi : ( (\phi < a)  \wedge   (\phi \not< b))),   \tag{AS3} 
\]\[
\forall \phi  \forall a \forall b  \,  (\phi < a \odot b \, \Leftrightarrow  (\phi < a) \vee (\phi < b)),   \tag{AS4} 
\]\[
\forall c \in C \,\, ((\phi < c) \Leftrightarrow (\psi < c)) \,  \Rightarrow (\phi = \psi),   \tag{AS5} 
\]\[
\forall a \exists \phi : (\phi < a).   \tag{AS6} 
\]
\end{definition} 
\bigskip

We use Greek letters to represent atoms and Latin letters to represent regular elements. Constants in $C$ are considered regular elements. We use ${\bf{C}}(t)$ for the component constants of a term $t$ (the constants mentioned in the term), and ${\bf{L}}^a_M(b)$ for the ``lower atomic segment'' of an element $b$ in the model $M$ which is the set of the atoms $\phi$ that satisfy $M \models (\phi < b)$ (see section \ref{Notation} for notation and definitions). The ``upper constant segment" ${\bf{U}}^c(\phi)$ of an atom $\phi$ is the set of constants $c \in C$ that satisfy $\phi < c$ in the model $M$. We drop the subindex $M$ and write ${\bf{U}}^c(\phi)$ instead of ${\bf{U}}^c_M(\phi)$ for reasons that will become apparent soon.
 
\bigskip
We are going to show first that the partial order of the atomized semilattice coincides with that of the semilattice spawned by its regular elements:  \[
\forall a \forall b \, (  a \leq b \,  \Leftrightarrow \, a \odot b = b).   \tag{AS3b} 
\] 
\bigskip

\begin{theorem} \label{axiomatizationProerties}  Assume AS4 and the antisymmetry of the order relation. \\
i) AS3 implies AS3b. \\
ii) Assume $\,\forall a \forall b  (\forall \phi ((\phi < a) \Leftrightarrow (\phi < b)) \, \Rightarrow (a = b))$. Then $AS3b \Rightarrow AS3.$  \\
iii) AS3 implies $\,\forall a \forall b  (\forall \phi ((\phi < a) \Leftrightarrow (\phi < b)) \, \Rightarrow (a = b))$.
\end{theorem} 
\begin{proof}
(i) $AS3 \Rightarrow AS3b$: Assume $a \leq b$. AS3 implies $\neg \exists \phi:((\phi < a)  \wedge  (\phi \not< b))$ and then $(\phi < a) \vee (\phi < b) \Leftrightarrow  (\phi < b)$. From here, using AS4, follows $(\phi < a \odot b) \Leftrightarrow (\phi < b)$, and we can use AS3 (right to left this time) to get $(b \leq a \odot b)  \wedge (a \odot b \leq b)$, and from the antisymmetry of the order relation, i.e. $(x = y) \Leftrightarrow (x \leq y) \wedge (y \leq x)$, we obtain $a = a \odot b$. \\
Assume $a \odot b = b$. Then $(\phi < a \odot b) \Leftrightarrow (\phi < b)$, and using AS4: $(\phi < a) \vee (\phi < b) \Leftrightarrow  (\phi < b)$ which implies $\neg \exists \phi : ((\phi < a)  \wedge   (\phi \not< b))$ from which, using AS3, we get $a \leq b$.  \\
(ii) $AS3b \Rightarrow AS3$: Assume $a \leq b$. AS3b implies $a \odot b = b$. From AS4 we get that $(\phi < a) \vee (\phi < b) \Leftrightarrow  (\phi < a \odot b = b)$, so $\phi < a$ implies $\phi < b$ and then $\neg \exists \phi : ((\phi < a)  \wedge   (\phi \not< b))$. \\
Assume now that $\neg \exists \phi : ((\phi < a)  \wedge   (\phi \not< b))$. Then $(\phi < a) \vee (\phi < b) \Leftrightarrow  (\phi < b)$ and AS4 leads to $(\phi < a \odot b) \Leftrightarrow  (\phi < b)$. However, we cannot go from here to $a \leq b$ unless we add an axiom such as: $\forall \phi ((\phi < a) \Leftrightarrow (\phi < b)) \, \Rightarrow (a = b)$.  \\
(iii) From $\forall \phi ((\phi < a) \Leftrightarrow (\phi < b))$ follows, using AS3, that $(a \leq b) \wedge (b \leq a)$, and by the antisymmetry of the order relation we get $a = b$.
\end{proof}
\bigskip
\bigskip

From theorem \ref{axiomatizationProerties}, AS3 and AS3b are equivalent modulus AS4 if the additional axiom: \[
\forall a \forall b  (\forall \phi ((\phi < a) \Leftrightarrow (\phi < b)) \, \Rightarrow (a = b)),
\]
is introduced. Since this additional axiom follows from AS3 and the antisymmetry of the order relation, we thus have two equivalent options: AS3b and the additional axiom or AS3 alone, and we choose the second option for economy of axioms.

\bigskip
\bigskip
\begin{theorem} \label{atomicSegmentFromTermTheorem}
Let $\,t,s \in F_C(\emptyset)$ be two terms that represent two regular elements $\nu_M(t)$ and $\nu_M(s)$ of an atomized model $M$ over a finite set of constants $C$. Let $\phi$ be an atom, $c$ a constant in $C$ and let $a$ be a regular element of $M$: \\
i) $\forall  t \forall c (c \in {\bf{C}}(t)  \, \Rightarrow \,  \nu_M(c) \leq \nu_M(t))$,\\
ii) $\phi < \nu_M(t) \, \Leftrightarrow  \exists  c: ((c \in {\bf{C}}(t)) \wedge  (\phi < \nu_M(c)))$,\\
iii) $(\phi < a) \, \Leftrightarrow  \exists  c : ((c \in C) \wedge (\phi < \nu_M(c) \leq a))$,\\   
iv) ${\bf{L}}^a_M(\nu_M(t))  = \{ \phi \in M: {\bf{C}}(t) \cap {\bf{U}}^c(\phi) \not= \emptyset \},$\\
v) ${\bf{L}}^a_M(\nu_M(s) \odot \nu_M(t)) \, = \,  {\bf{L}}^a_M(\nu_M(t)) \cup {\bf{L}}^a_M(\nu_M(s))$,\\
vi) $\nu_M(t) \leq \nu_M(s) \, \Leftrightarrow \,  {\bf{L}}^a_M(\nu_M(t))  \subseteq {\bf{L}}^a_M(\nu_M(s)).$
\end{theorem}
\begin{proof}
(i) From $t = t \odot c$ and the natural homomorphism $\nu_M(t) = \nu_M(t \odot c) = \nu_M(t) \odot \nu_M(c)$ we get $\nu_M(c) \leq \nu_M(t)$.   \\
(ii) Right to left, $\phi < \nu_M(c) \leq \nu_M(t)$ follows from i and, from and the transitivity of the order relation, $\phi < \nu_M(t)$. Left to right can be proven from the fourth axiom of atomized models $\phi < a \odot b \, \Rightarrow  (\phi < a) \vee (\phi < b)$ applied to the component constants ${\bf{C}}(t)$ of $t$. The number of component constants of $t$ is at least 1 and at most $\vert C \vert$ so it is a finite number and we need to apply this axiom a finite number of times to get $\phi < \nu_M(c)$ for some component constant of $t$. This proves (ii), and (iii) is a consequence of (ii) that follows with just choosing any term $t$ of $F_C(\emptyset)$ to represent element $a = \nu_M(t)$.  \\
(iv) Consider an atom $\phi \in {\bf{L}}^a_M(\nu_M(t))$ then $\phi < \nu_M(t)$ and from proposition (ii) there is a component constant $c \in {\bf{C}}(t)$ such that $\phi < \nu_M(c)$ which means that $c \in {\bf{C}}(t) \cap {\bf{U}}^c(\phi)$. Conversely, if $c \in {\bf{C}}(t) \cap {\bf{U}}^c(\phi)$ then $\phi < \nu_M(c) \leq \nu_M(t)$ and $\phi \in {\bf{L}}^a_M(\nu_M(t))$.  \\
(v) Since $\nu_M$ is a homomorphism $\nu_M(s) \odot \nu_M(t) = \nu_M(s \odot t)$ and, using proposition (iv) ${\bf{L}}^a_M(\nu_M(s \odot t))  = \{ \phi \in M: {\bf{C}}(s \odot t) \cap {\bf{U}}^c(\phi) \not= \emptyset \} = \{ \phi \in M: ({\bf{C}}(s) \cup {\bf{C}}(t)) \cap {\bf{U}}^c(\phi) \not= \emptyset \} = {\bf{L}}^a_M(\nu_M(t)) \cup {\bf{L}}^a_M(\nu_M(s))$. Note that this proposition is an alternative way to write axiom AS4.   \\
(vi) It is straightforward from proposition v and AS3b or AS3. 
\end{proof}	
\bigskip
\bigskip

The first axiom of atomized semilattices says that the upper constant segment of an atom is never empty while the fifth axiom says that two distinct atoms cannot have the same upper constant segment.

Axiom AS3 implies that the order relation for regular elements is encoded and fully determined by the atoms of the model. Theorem \ref{atomicSegmentFromTermTheorem} propositions (iv) to (vi) show that it suffices with knowing the constants in the upper segments of each atom in the atomized model to know the entire model. Furthermore, it is enough with knowing the atoms of a model $M$ and the component constants ${\bf{C}}(t)$ of a term to know the atoms in the lower segment of $t$ in $M$.

Theorem \ref{atomicSegmentFromTermTheorem} proposition (v) proves the linearity property: \[
{\bf{L}}^a_M(s \odot t) \, = \,  {\bf{L}}^a_M(t) \cup {\bf{L}}^a_M(s),
\]
where we have omitted the natural homomorphism.

Theorem \ref{atomicSegmentFromTermTheorem} proposition (vi) and AS3b imply that \[ M \models (t = s) \, \Leftrightarrow \, {\bf{L}}^a_M(t) =  {\bf{L}}^a_M(s).
\] Therefore, it does not matter which term we use to represent a regular element we get the same atoms using: \[
{\bf{L}}^a_M(t) = \{ \phi \in M: {\bf{C}}(t) \cap {\bf{U}}^c(\phi) \not= \emptyset \}.
 \]
 
In order to use ${\bf{U}}^c(\phi)$ instead of ${\bf{U}}^c_M(\phi)$ we just need to assume, without loss of generality, that atoms of different atomized semilattice models that have the same upper constant segments also have the same name. With this choice an atom can be defined by a set ${\bf{U}}^c(\phi)$, independently of other atoms and independently of the model it belongs to. We can define the atoms by their upper constant segments ${\bf{U}}^c(\phi)$ as well as the models by the atoms they have, as follows:

\bigskip
\bigskip
\begin{definition} \label{atomInAsAsetDefinition}  
Consider a non-empty subset ${\bf{U}}^c(\phi)$ of a set of constants $C$. For any atomized semilattice $M$ over $C$, any set $A$ of atoms that atomizes $M$ and any constant $c$ in $C$, the atom $\phi$ defined by ${\bf{U}}^c(\phi)$ is an atom that satisfies $M \models (\phi < c)$ if and only if $\phi \in A$ and $c \in {\bf{U}}^c(\phi)$. 
\end{definition} 
\bigskip
\bigskip
Since atoms are non-empty sets of constants any atomization of a semilattice model is a hypergraph with the atoms as hyperedges and the constants as vertexes.

Axiom AS3 can be rewritten to show the connection between universally defined atoms and models as: 
\[
(t \leq_M s) \,  \Leftrightarrow  \, \neg \exists \phi: ((\phi \in M) \wedge  (\phi < t)  \wedge (\phi \not< s)).
\]
where $t$ and $s$ are terms and $t \leq_M s$ means $M \models  (t \leq s)$ and the natural homomorphism has been omitted. We add the subscript in $\leq_M$ to highlight the difference with $<$ that compares atoms and terms independently of the model. We do not use subscripts and simply write $\leq$ for $\leq_M$ when there is no ambiguity.

The next theorem permits to identify ${\bf{U}}^c(\phi)$ with the upper constant segment of $\phi$ in the freest model $F_C(\emptyset)$, as well as $\leq$ with $\leq_{F_C(\emptyset)}$.

We sometimes use $[A]$ to refer to the model atomized by a set of atoms $A$, and sometimes use the same letter $M$ to refer to both, the model and the set of atoms. 

\bigskip

\begin{theorem} \label{freestModelCanBeAtomizedByAllTheAtoms}
If $A$ is a set of atoms that atomizes $F_C(\emptyset)$ and $\psi$ is any atom, the set $A \cup \{ \psi \}$ is also an atomization for $F_C(\emptyset)$.
\end{theorem} 
\begin{proof}
Let $t$, $s$ be two terms. In any semilattice model, ${\bf{C}}(t) \subseteq {\bf{C}}(s)$ implies $t \leq s$. The freest semilattice is characterized by the double implication: $t \leq s$ if and only if ${\bf{C}}(t) \subseteq {\bf{C}}(s)$.  \\
An atom defined by its non-empty set ${\bf{U}}^c(\phi)$ according to definition \ref{atomInAsAsetDefinition} necessarily satisfies the axioms AS1 and is consistent with axioms AS2, AS4 and AS5. Since $[A]$ satisfies axiom AS6 then $[A \cup \{ \psi \}]$ also satisfies axiom AS6. Consider the axiom AS3: $\forall a \forall b \, (  a \leq b \,  \Leftrightarrow  \neg \exists \phi : ( (\phi < a)  \wedge  (\phi \not< b)))$.  If ${\bf{C}}(t) \subseteq {\bf{C}}(s)$, from theorem \ref{atomicSegmentFromTermTheorem}(iv) follows that $(\psi < t) \Rightarrow (\psi < s)$. If ${\bf{C}}(t) \not\subseteq {\bf{C}}(s)$ then $F_C(\emptyset) = [A] \models (t \not< s)$ and the clause $\neg \exists \phi : ( (\phi < t)  \wedge  (\phi \not< s)))$ does not hold. In both cases, ${\bf{C}}(t) \subseteq {\bf{C}}(s)$ and ${\bf{C}}(t) \not\subseteq {\bf{C}}(s)$, if $[A]$ satisfies the clause $\neg \exists \phi : ( (\phi < t)  \wedge  (\phi \not< s)))$ then $[A \cup \{ \psi \}]$ also does, so $t \leq s$ holds in $[A \cup \{ \psi \}]$ if and only if it holds in $[A]$ and it follows that the semilattice formed by the regular elements of $[A \cup \{ \psi \}]$ is equal to $F_C(\emptyset)$.
\end{proof}	
\bigskip
\bigskip

A corollary of theorem \ref{freestModelCanBeAtomizedByAllTheAtoms} is that the set of all possible atoms over $C$ is an atomization of $F_C(\emptyset)$. We can say that every atom ``is in'' the freest semilattice $F_C(\emptyset)$ and write $\phi \in F_C(\emptyset)$ for any atom $\phi$ defined by a proper subset of $C$. We will see later that there are atomizations for the freest semilattice with fewer atoms. 

\bigskip
\begin{definition} \label{isInDefinition}  
Let $M = [A]$ be the semilattice spawned by the atoms $A$.  We say an atom $\phi$ is in $M$, written $\phi \in M$, if the semilattice formed by the regular elements of $[A \cup \{ \phi \}]$ is equal to $M$. 
\end{definition} 

\begin{definition} \label{isInForSemilatticesDefinition}  
Let $M$ be a semilattice.  We say an atom $\phi$ is in $M$, written $\phi \in M$, if there is an atomized semilattice spawned by some set of atoms that contains $\phi$ with regular elements forming a semilattice equal to $M$. 
\end{definition} 
\bigskip
\bigskip

Since an atom is universally defined by a non-empty subset of $C$, by using definition \ref{atomInAsAsetDefinition} it is possible to identify the same atom in different models and to compare the same atom with elements of different models provided that the models are generated by the same constants. In the next two theorems we compare the same atoms with different models: $F_C(\emptyset)$, $M$ and $N$.

\bigskip
\begin{theorem} \label{homomorphismTheorem}
Let $t,s \in F_C(\emptyset)$ be two terms and $M$ an atomized semilattice model. Let $\phi$ be an atom and $\nu_M : F_C(\emptyset) \rightarrow M$ the natural homomorphism.
\[
i)\,\,\, (F_C(\emptyset) \models (t \leq  s)) \,\, \Rightarrow \,\,  M \models (\nu_M(t) \leq \nu_M(s)) 
\]\[
ii)\,\,\, (\phi \in M ) \wedge (F_C(\emptyset) \models (\phi < s)) \,\, \Leftrightarrow \,\,  M \models (\phi < \nu_M(s)) 
\]
\end{theorem} 
\begin{proof}
i) Follows immediately from the fact that $\nu_M$ is a homomorphism and is provided here for comparison with proposition (ii). \\ 
ii) Note that we use here the same atom $\phi$ in the contexts of two atomized models, $F_C(\emptyset)$ and $M$. Left to right: using theorem \ref{atomicSegmentFromTermTheorem}(iii), $F_C(\emptyset) \models (\phi < s)$ implies that there is some constant $c$ such that $F_C(\emptyset) \models (\phi < c \leq s)$ and then, from the natural homomorphism, $M \models (\nu_M(c) \leq \nu_M(s))$.  From definition \ref{atomInAsAsetDefinition}, $F_C(\emptyset) \models (\phi < c)$ requires $c \in {\bf{U}}^c(\phi)$ and then, because we assume $\phi \in M$, the same definition implies $M \models (\phi < \nu_M(c))$.  Using the transitive property of the order relation $M \models (\phi < \nu_M(s))$.  \\ 
Right to left can be proven with the same argument than left to right interchanging the roles of $M$ and $F_C(\emptyset)$ except for the fact that we do not need to require $\phi \in F_C(\emptyset)$ as this is always the case for any atom, as proven in theorem \ref{freestModelCanBeAtomizedByAllTheAtoms}.
\end{proof}	
\bigskip
\bigskip

Note that the implication arrow in proposition (i) of the last theorem goes only left to right while in proposition (ii) goes in both directions. 

We say that an atom $\phi$ ``discriminates'' a duple $r = (r_L, r_R)$ if $(\phi < r_L) \wedge (\phi \not< r_R)$. If $\phi \in M$ then axiom AS3 implies that $M \models r^{-}$.

\bigskip
\bigskip
\begin{theorem} \label{atomIndependentFromTheRestTheorem}
Let $t,s \in F_C(\emptyset)$ be two terms and $M$ and $N$ two models. If an atom $\phi$ discriminates a duple $(t ,s)$ in $M$ then, either $\phi \not\in N$ or $\phi$ also discriminates $(t ,s)$ in $N$.
\end{theorem} 
\begin{proof}
If $\phi$ discriminates $(t ,s)$ in $M$, using part (ii) of theorem \ref{homomorphismTheorem} right-to-left for both terms, follows that $\phi$ discriminates $(t,s)$ in $F_C(\emptyset)$ and then, using again theorem \ref{homomorphismTheorem}(ii) left-to-right, we get that either $\phi \not\in N$ or $\phi$ also discriminates $(t,s)$ in $N$. 
\end{proof}	
\bigskip

\bigskip
\begin{definition} \label{wider}  
We say an atom $\phi$ is \emph{wider} than atom $\eta$ if it is different than $\eta$ and for every constant $c$, $(\eta < c) \Rightarrow (\phi < c)$.
\end{definition} 
\bigskip

 Equivalently, an atom $\phi$ is \emph{wider} than atom $\eta$ when ${{\bf{U}}^{c}}(\eta) \subsetneq {{\bf{U}}^{c}}(\phi)$. It is possible to further extend the partial order by taking $\phi < \eta$ if and only if $\phi$ is wider than $\eta$. In this text atoms are only compared with regular elements and such extension is not needed.

\bigskip
\begin{definition} \label{redundantAtomDefi}  
We say an atom $\phi$ is redundant with $M$ when for each constant $c$ such that $\phi < c$ there is at least one atom $\eta < c$ in $M$ such that $\phi$ is wider than $\eta$. 
\end{definition} 
\bigskip

With respect to a model $M$ we say that an atom $\phi$ is ``external to $M$" if $\phi$ is not in $M$ (as defined in \ref{isInDefinition}) and we write $\phi \not\in M$. If an atom is in $M$ it may be either redundant with $M$ or not. If the atom is in $M$ and is not redundant with $M$ then we call it a ``non-redundant atom in $M$''. Therefore, with respect to a model $M$ an atom can be either a ``non-redundant atom in $M$'' or ``redundant with $M$'' or ``external to $M$''.  

We will see in theorem \ref{redundantAtom} that redundant atoms can be discarded without altering the semilattice model. In addition, in theorem \ref{uniqueAtomization} we prove that the set of atoms that are non-redundant with $M$ is unique, and there is only one atomization of $M$ with every atom non-redundant in $M$. As a consequence of these theorems the non-redundant atoms in $M$ are those that are on every possible atomization of $M$ while the redundant atoms with $M$ are those that are only in some atomizations of $M$. When we work with an specific set of atoms $A$, a redundant atom with $M$ may or may not be in set $A$. 

Atomized semilattice models often have many atoms. We are interested in the subset of non-redundant atoms that suffices to atomize the model. For example, in theorem  \ref{freestModelTheorem} we will show that $F_C(\emptyset)$ can be atomized with as few as $\vert C \vert$ non-redundant atoms, however, every one of the $2^{\vert C\vert} - 1$ possible atoms over $C$ is an atom in $F_C(\emptyset)$ as proven in theorem \ref{freestModelCanBeAtomizedByAllTheAtoms}. The number of atoms in a model $M = [A]$ atomized by a set $A$ of atoms is usually much larger than $\vert A\vert$ as most atoms might be redundant.

\bigskip
\bigskip

The zero atom, $\ominus_c$, is defined by $\,{\bf{U}}^{c}(\ominus_c) = C\,$ and satisfies $\ominus_c < d$ for each constant $d \in C$. We can always add to any atomized model the atom $\ominus_c$ without changing the semilattice formed by its regular elements. If a set of atoms $A$ spawns a model $[A]$ then the semilattice spawned by $[A  \cup  \{\ominus_c\}]$ is the same semilattice than $[A]$. 

If we use definition \ref{atomInAsAsetDefinition} a set $A$ of atoms always forms an atomized semilattice if it satisfies the sixth axiom. If $A$ does not form and atomized semilattice then theorem \ref{atomizationsAreModels} tells us that $A  \cup  \{\ominus_c\}$ does.

Atom $\ominus_c$ is redundant in most models. It is redundant unless there is at least one constant $q$ such that $M \models \forall x (q \leq x)$, a constant that acts as a neutral element. If there are two such constants, $q_1$ and $q_2$, then $M \models (q_1 = q_2)$.

\bigskip
\bigskip
\begin{theorem} \label{ceroTheorem}
Let $M$ be a semilattice over $C$. \\
i) $\ominus_c$ is in $M$,  \\
ii) $\ominus_c$ is a non-redundant atom in $M$ if and only if there is at least one constant $q \in C$ such that $M \models \forall x (q \leq x)$.
\end{theorem}
\begin{proof}
(i) Suppose $M$ is atomized by a set of atoms $A$. Since $F_C(\emptyset) \models \forall t (\ominus_c < t)$ then, from theorem \ref{homomorphismTheorem}(ii), the model spawned by $A \cup \{\ominus_c\}$ satisfies $ \forall t (\ominus_c < t)$. Since an atom in the lower segment of every regular element cannot discriminate any duple, then $[A \cup \{\ominus_c\}] = M$ and it follows that $\ominus_c \in M$. \\
(ii) By definition, $\ominus_c$ is redundant with $M$ if for each constant $c \in C$ there is at least one atom $\eta \in M$ different than $\ominus_c$ such that $\eta < c$. It follows that $\ominus_c$ is not redundant with $M$ if and only if there is some constant $q \in C$ for which there is no other atom $\eta \in M$ such that $\eta < q$. We have shown in proposition (i) that $\ominus_c \in M$ so if $\ominus_c$ is not redundant with $M$ then $\ominus_c$ is a non-redundant atom in $M$. Therefore, if $\ominus_c$ is a non-redundant atom in $M$ there is at least one constant $q$ that only has $\ominus_c$ in its lower atomic segment. Since $\ominus_c$ is in the lower segment of every element of $M$, theorem \ref{atomicSegmentFromTermTheorem}(iii) implies that $M \models \forall x (q \leq x)$.  \\
On the other direction, if there is a constant $q$ such that $M \models \forall x (q \leq x)$ then $q \leq c$ for each constant $c$. Every atom $\phi < q$ satisfies $\phi < c$ for every constant of $C$, therefore, ${\bf{U}}^{c}(\phi) = C$ and axiom AS5 implies $\phi = \ominus_c$. In addition, $\ominus_c$ cannot be redundant with $M$ because if it were, there would be an atom $\eta < q$ with $\ominus_c$ wider than $\eta$ contradicting that $q$ only has $\ominus_c$ in its lower segment.
 \end{proof}

\bigskip 
\bigskip
\begin{theorem} \label{atomizationsAreModels}
Let $A$ be a set of atoms over the constants $C$ with each atom defined by its upper constant segment according to \ref{atomInAsAsetDefinition}. Either $A$ or $A \cup \{\ominus_c\}$ spawns an atomized semilattice model.  
\end{theorem} 
\begin{proof}
To construct the atomized semilattice model $M = [A]$ spawned by a set of atoms $A$, we start with the relation $\leq_M$ over $F_C(\emptyset)$ defined, for any two terms $t,s \in F_C(\emptyset)$, by $(t \leq_M s) \equiv  \neg \exists \phi:(\phi \in A) ((\phi < t)  \wedge (\phi \not< s))$, where $(\phi < t) \equiv ({\bf{C}}(t) \cap {\bf{U}}^c(\phi) \not= \emptyset)$. The elements of $M$ correspond to equivalence classes in $F_C(\emptyset)$ with $t$ and $s$ in the same class if and only if $(t \leq_M s)  \wedge (s \leq_M t)$. The natural homorphism $\nu_M(t)$ maps each term to its equivalence class and $(\nu_M(t) = \nu_M(s)) \equiv (t \leq_M s)  \wedge (s \leq_M t)$. The idempotent summation $a \odot b$ for two elements $a$ and $b$ of $M$ can be built by choosing any terms in the inverse images of the natural homorphism, say $t_a$ and $t_b$ with $a = \nu_M(t_a)$ and $b = \nu_M(t_b)$ and by calculating $\nu_M(t_a \odot t_b)$. With the help of the idempotent operator $\odot$ we can extend the relation $\leq_M$ defined for $F_C(\emptyset)$ to a relation in $M$ using the usual $M \models (a \leq b)$ iff $M \models (b = a \odot b)$. \\ 
To show that the idempotent summation of elements is consistently defined, suppose we choose other terms $s_a, s_b$ such that $a = \nu_M(s_a)$ and $b = \nu_M(s_b)$. Then $(\nu_M(t_a) = a = \nu_M(s_a))$ and $(\nu_M(t_b) = b = \nu_M(s_b))$ and these sentences are equivalent to $(t_a \leq_M s_a)  \wedge (s_a \leq_M t_a)$ and $(t_b \leq_M s_b)  \wedge (s_b \leq_M t_b)$ respectively. Using the definition of $\leq_M$ in $F_C(\emptyset)$ given above it is straightforward to show that these two sentences are true if and only if $\{ \phi: (\phi \in A) \wedge (\phi < t_a)\} = \{ \phi: (\phi \in A) \wedge (\phi < s_a)\}$ and $\{ \phi: (\phi \in A) \wedge (\phi < t_b)\} = \{ \phi: (\phi \in A) \wedge (\phi < s_b)\}$ are true. From the definition of $(\phi < t)$ and ${\bf{C}}(t \odot s) = {\bf{C}}(t) \cup {\bf{C}}(s)$ we can write: 
\[
\{ \phi: (\phi \in A) \wedge (\phi < t_a  \odot t_b)\} =  \{ \phi: (\phi \in A) \wedge (\phi < t_a)\} \cup  \{ \phi: (\phi \in A) \wedge (\phi < t_b)\} =
\]\[
= \{ \phi: (\phi \in A) \wedge (\phi < s_a)\} \cup  \{ \phi: (\phi \in A) \wedge (\phi < s_b)\} = 
\]\[
 = \{ \phi: (\phi \in A) \wedge (\phi < s_a  \odot s_b)\} 
\]
and this proves $(t_a \odot t_b \leq_M s_a \odot s_b)  \wedge (s_a \odot s_b \leq_M t_a \odot t_b)$ or, in other words: \[ 
\nu_M(t_a \odot t_b) = \nu_M(s_a \odot s_b).
\]  
We have shown that it does not matter which terms of $F_C(\emptyset)$ we choose to represent $a$ and $b$ the element of $M$ defined by $\nu_M(t_a \odot t_b)$ is the same. Since the union of sets is idempotent, commutative and associative and ${\bf{C}}(t \odot s) = {\bf{C}}(t) \cup {\bf{C}}(s)$ is easy to see that our operation $\odot$ is commutative, associative and idempotent in $M$ and we can identify the structure built with a semilattice.  \\ 
It is straightforward to check that the other axioms of atomized semilattices also hold with the exception of the last one: $\forall a \exists \phi \, (\phi < a)$, that may or may not hold for $A$. Since $\forall a (\ominus_c < a)$, the algebraic structure spawned by the atoms $A  \cup  \{\ominus_c\}$ satisfies $(t \leq_{[A]} s) \Leftrightarrow (t \leq_{[A  \cup  \{\ominus_c\}]} s)$ and also satisfies the sixth axiom of atomized semilattices. Therefore, $[A  \cup  \{\ominus_c\}]$ is an atomized semilattice model. 
\end{proof}	
\bigskip
\bigskip

We have shown that any atomization spawns a semilattice and we show in theorem \ref{atomizationExistsTheorem} that all semilattice models can be atomized. 

Theorem \ref{homomorphismTheorem} provides a natural way to extend the natural homorphism from semilattices to atomized semilattices. The extended natural homomorphism maps regular elements of $F_C(\emptyset)$ to regular elements of $M$ and is equal to the semilattice natural homomorphism in this case. For atoms, the extended natural homomorphism maps atoms to atoms, as follows:
\[
\nu_M(\phi) =  \begin{dcases}
        \phi & \phi \in M \\
        \ominus_c & \phi \not\in M \\
    \end{dcases} 
\]
which works because every atom is in the freest semilattice $F_C(\emptyset)$, and $\ominus_c$ is in every model.

We made the choice of using the elements of $F_C(\emptyset)$ as elements names for every model because terms provide references that are universal, i.e. references that apply to any model over $C$ and can be transferred from model to model. We have seen that atoms also have the same kind of universality. An atom can be compared with every regular element of any model and has meaning in the context of any model. Atoms of a model can be carried out to other atomizations to build models with properties of interest. 
\newpage

\section{Full crossing, freedom and redundancy} \label{freedomFullCrossingAndRedundacy}

\begin{definition} \label{freerModelDefinition}  
Let $A$ and $B$ be two models over the constants $C$. We say that model $A$ is ``freer or as free" than model $B$ if for every duple $r$ over $C$ we have $\,\,B \models r^{-} \Rightarrow A \models r^{-}$ 
\end{definition} 
\bigskip

\begin{theorem} \label{atomsCanBeRemovedTheorem}
Let $C$ be a set of constants, $K \subseteq C$ and $r$ a duple over $K$. Let $M$ be an atomized semilattice model over $C$ and $A$ a set of atoms that atomizes $M$. Let $B$ be a subset of $A$ and $N = [B  \cup  \{  \ominus_K \}]$ an atomized semilattice model over $K$. \\ 
i) If $M \models r^{+}$ then $N \models r^{+}$.  \\ 
ii)  $M$ is as free or freer than $N$.
\end{theorem}
\begin{proof}
Let $r = (r_L, r_R)$, the duple over $K$. Axiom AS3 of atomized semilattices in $N$ reads $\nu_N(r_L) \leq \nu_N(r_R)  \,  \Leftrightarrow  \neg \exists \phi: ((\phi  \in B \cup  \{  \ominus_K \}]) \wedge (\phi < r_L)  \wedge (\phi \not< r_R ))$ while, for model $M$ is $\nu_M(r_L) \leq \nu_M(r_R)  \,  \Leftrightarrow  \neg \exists \phi: ((\phi \in A) \wedge  (\phi < r_L)  \wedge   (\phi \not< r_R ))$. Since $\ominus_K$ discriminates no duple over $K$ and $B \subseteq A$, it follows that $(\nu_M(r_L) \leq \nu_M(r_R))  \Rightarrow  (\nu_N(r_L) \leq \nu_N(r_R))$. Therefore, a model $N$ spawned by any subset of atoms of $M$ satisfies all the positive duples satisfied by $M$. \\
By negating the sentence above we get $(\nu_N(r_L) \not\leq \nu_N(r_R))  \Rightarrow (\nu_M(r_L) \not\leq \nu_M(r_R))$ which means that $M$ satisfies all the negative duples of $N$ and, hence, $M$ is as free or freer than $N$.
\end{proof}	

\bigskip

We need a few definitions:

The discriminant ${\bf{dis}}_M(a,b)$ in an atomized model $M$ is the set ${{\bf{L}}^{a}_M}(a) - {{\bf{L}}^{a}_M}(b)$ of atoms of $M$. From theorem \ref{atomicSegmentFromTermTheorem}(iv) follows that $M \models a \leq b$ if and only if ${\bf{dis}}_M(a,b)$ is empty. 

For each atom $\phi \not= \ominus_c$, let the ``pinning term'' $T_{\phi}$ be equal to the idempotent summation of all the constants in the set $C - {\bf{U}}^{c}(\phi)$. For each constant $c \in {\bf{U}}^{c}(\phi)$, i.e, for each constant such that $(\phi < c)$, define the ``pinning duple'' as $\neg(c < T_{\phi})$.  $PR(\phi)$ is the set of pinning duples of an atom $\phi$. Note that pinning terms and pinning duples of an atom are independent of the model as they only depend upon $C$ and the set ${\bf{U}}^{c}(\phi)$. 

Let $Th_0(M)$ be the set of duples satisfied by (the regular elements of) the model $M$, and the positive and negative theories of $M$, written $Th_0^{+}(M)$ and $Th_0^{-}(M)$ respectively, are the sets of positive duples satisfied by $M$ and the set of negative duples satisfied by $M$.

\bigskip
\bigskip

The next theorem shows that an atom that is redundant with a model $M$ atomized by a set of atoms $A$ can be added to $A$ or taken out from $A$ without changing the model.

\bigskip
\bigskip

\begin{theorem} \label{redundantAtom}
Let $A$ be a set of atoms that atomizes a model $M$ and $\phi$ an atom of $A$. The set $A - \{\phi\}$ also atomizes $M$ if and only if $\phi$ is redundant with $M$.
\end{theorem}
\begin{proof}
Let $Th_0^{+}(M)$ be the set of all positive duples satisfied by the regular elements of $M$. Theorem \ref{atomsCanBeRemovedTheorem} tells us that positive duples do not become negative when atoms are taken out of an atomization. Taking out an atom $\phi$ from a model $M$ produces a model $N \models Th_0^{+}(M)$. Therefore, when removing an atom we only need to worry about negative duples that may become positive. To prove that a redundant atom can be eliminated, let $a$ and $b$ be a pair of terms of $F_C(\emptyset)$ and suppose $M \models (a \not\leq b)$, a negative duple satisfied by $M$ and discriminated by an atom $\phi < c \leq a$ where $c$ is some constant. If $\phi$ is redundant there is an atom $\eta < c$ in $M$ such that $\phi$ is wider than $\eta$. Suppose $\eta < b$. There is a constant $e$ such that $\eta < e \leq b$. Because $\phi$ is wider than $\eta$ then $\phi < e \leq b$ which contradicts the assumption that $\phi \in {{\bf{dis}}_M}(a, b)$. Therefore, $N \models (\eta \not< b)$, and then $\eta \in {{\bf{dis}}_N}(a, b)$ and $N \models (a \not\leq b)$. It follows that any negative duple of $M$ is also negative in $N$ and, if $N$ models the same positive and negative duples than $M$, then the semilattices $M$ and $N$ are equal. \\
Conversely, assume atom $\phi$ can be eliminated without altering $M$. Also assume $\phi \not= \ominus_c$. The pinning term $T_{\phi}$ is the idempotent summation of all the constants in the set $C - {\bf{U}}^{c}(\phi)$.  For each constant $c$ such that $\phi < c$ we have $\phi \in {{\bf{dis}}_M}(c, T_{\phi})$. From $F_C(\emptyset) \models (\phi \not< T_{\phi})$ and theorem \ref{atomIndependentFromTheRestTheorem} follows that $\phi$ discriminates duple $(c, T_{\phi})$ in $M$. If $\phi$ can be eliminated without altering $M$, for each constant $c$ in ${\bf{U}}^{c}(\phi)$ there should be some other atom $\eta_c$ in $M$ with $\eta_c < c$ and $\eta_c \not= \phi$ discriminating the duple $(c, T_{\phi})$ which implies $T_{\phi} < T_{\eta_c}$ or, equivalently, that $\phi$ is wider than $\eta_c$. Hence, for each constant such $\phi < c$ there is an $\eta_c \in M$ such that $\phi$ is wider than $\eta_c$ which proves $\phi$ is redundant.  \\ 
Finally, suppose that $\phi = \ominus_c$ is a not-redundant atom in $M$. Then theorem \ref{ceroTheorem} tells us that there is a constant $q \in C$ such that $M \models \forall x (q \leq x)$, a constant $q$ with a lower atomic segment that contains only a single atom: $\phi$ (see the proof of  \ref{ceroTheorem}). If $\phi$ is taken out of the model there would be a constant with no atoms in its lower segment, which contradicts the sixth axiom of atomized semilattices. Hence,  $A - \{\phi\}$ does not atomize $M$.
\end{proof}	
\bigskip

\begin{theorem} \label{uniqueRelationTheorem}
Let $M$ be a model and $\phi \not= \ominus_c$ a non-redundant atom of $M$. There is at least one pinning duple that is discriminated by $\phi$ and only by $\phi$. 
\end{theorem}
\begin{proof}
As in the proof of theorem \ref{redundantAtom}, for each constant $c \in {\bf{U}}^{c}(\phi)$ we have $\phi \in {{\bf{dis}}_M}(c, T_{\phi})$ where $T_{\phi}$ is the pinning term of $\phi$. This is a consequence of theorem \ref{atomIndependentFromTheRestTheorem} and $F_C(\emptyset) \models (\phi \not< T_{\phi})$. If, for a pinning duple $(c, T_{\phi})$ there is an atom $\varphi \not= \phi$ of $M$ that discriminates $(c, T_{\phi})$ then $\phi$ is wider than $\varphi$ and, if the same is true for every pinning duple in  $PR(\phi)$, then $\phi$ is redundant with $M$, which is against our assumptions. Therefore, there should be at least one constant $c$ such that $(c, T_{\phi})$ is discriminated only by $\phi$.
\end{proof}	
\bigskip
\bigskip

We introduce now the union of atoms, $\phi \bigtriangledown \psi$, that is an an atom with an upper constant segment ${\bf{U}}^c(\phi \bigtriangledown \psi) = {\bf{U}}^c(\phi) \cup {\bf{U}}^c(\psi)$.

\bigskip
\begin{theorem} \label{unionProperties}
Let $x$ be a regular element of an atomized semilattice model M and let $\phi$, $\psi$ and $\eta$ be atoms of M. The union of atoms has the properties:
\[
i) \,\, \phi \bigtriangledown \phi = \phi, 
\]\[
ii) \,\, \bigtriangledown \text{ is commutative and associative},  
\]\[
iii)  \,\, \phi < x  \Rightarrow  (\phi \bigtriangledown \psi < x), 
\]\[
iv)  \,\, (\phi \bigtriangledown \psi < x)  \Leftrightarrow (\phi < x) \vee  (\psi < x),
\]\[
v)  \,\, (\phi \bigtriangledown \psi < x) \wedge (\phi \not< x)  \Rightarrow (\psi < x).
\]\[
vi)  \,\, \phi \text{ is wider or equal to } \eta \text{ if and only if }  \phi = \phi \bigtriangledown \eta.
\]
\end{theorem}
\begin{proof}
According to definition \ref{atomInAsAsetDefinition} an atom is determined by the constants in its upper segment, therefore atom $\phi \bigtriangledown \psi$ is fully defined by ${\bf{U}}^c(\phi \bigtriangledown \psi) = {\bf{U}}^c(\phi) \cup {\bf{U}}^c(\psi)$. Propositions i and ii follow from the idempotence, commutativity and associativity of the union of sets.  Let $t$ be a term of $F_C(\emptyset)$ and let $\nu_M$ be the natural homomorphism of $F_C(\emptyset)$ onto $M$. Theorem \ref{atomicSegmentFromTermTheorem} says that $ {\bf{L}}^a_M(\nu_M(t))  = \{ \phi \in M: {\bf{C}}(t) \cap {\bf{U}}^c(\phi) \not= \emptyset \},
$ which permits to calculate the lower atomic segment of an element $\nu_M(t)$, represented with any term $t$, by using the component constants of the term.  Let $t$ be any term such that $x = \nu_M(t)$:  \\
(iii) $\phi < x$ implies that exists $c \in {\bf{C}}(t)$ such that $\phi < c$ and, hence, $c \in {\bf{U}}^c(\phi) \subseteq {\bf{U}}^c(\phi \bigtriangledown \psi)\,$ so $\,\phi \bigtriangledown \psi  < c \leq x$.  \\
(iv) Right to left, is the same than proposition (iii). Left to right; the left hand side can be written as $\exists t \exists c ((x = \nu_M(t)) \, \wedge \, c \in {\bf{C}}(t) \wedge (\phi \bigtriangledown \psi < c))$, which implies $c \in {\bf{U}}^c(\phi) \cup {\bf{U}}^c(\psi)$ and then $(\phi < c) \vee (\psi < c)$ that, together with $c \leq x$, yields $(\phi < x) \vee (\psi < x)$.  \\ 
(v) is a trivial (but useful) consequence of proposition (iv).  \\ 
(vi) $\phi$ is wider than, or equal to, atom $\eta$ if and only if for every constant $c$, $(\eta < c) \Rightarrow (\phi < c)$ or, in other words, ${{\bf{U}}^{c}}(\eta) \subseteq {{\bf{U}}^{c}}(\phi)$. It follows that ${\bf{U}}^c(\phi \bigtriangledown \eta) = {\bf{U}}^c(\phi) \cup {\bf{U}}^c(\eta) = {\bf{U}}^c(\phi)$. Hence, $\phi$ is wider than or equal to atom $\eta$ if and only if $\phi = \phi \bigtriangledown \eta$.  
\end{proof}
\bigskip

\begin{theorem} \label{redundantIsUnionOfAtomsTheorem}
Let $M$ be an atomized model over a finite set $C$ of constants. Let $\phi$ be an atom that may or may not be in the atomization of $M$. \\
i) $\phi$ is redundant with $M$ if and only if is a union $\phi = \bigtriangledown_{i} \eta_{i}$ of atoms of $M$ such that $\forall i (\phi \not= \eta_{i})$. \\
ii) $\phi$ is redundant with $M$ if and only if is a union of two or more non-redundant atoms of $M$.
\end{theorem}
\begin{proof}
If $\phi$ is redundant for each constant such $that \phi < c$ there is an atom $\eta_i$ of $M$ such that $\phi$ is wider than $\eta_i$ and $\eta_i < c$. Theorem \ref{unionProperties}(vi) states that if $\phi$ is wider than $\eta_i$ then $\phi = \phi \bigtriangledown \eta_i$ and, since for each constant $c \in {{\bf{U}}^{c}}(\phi)$ there is some $\eta_i$ such that $\eta_i < c$ then ${{\bf{U}}^{c}}(\phi) = \cup_{i}{{\bf{U}}^{c}}(\eta_{i})$. It follows $\phi = \bigtriangledown_{i} \eta_{i}$. Conversely if $\phi$ is a union of atoms, $\phi = \bigtriangledown_{i} \eta_{i}$, then for each constant such $\phi < c$ there is some atom $\eta_{i}$ that contains $c$ in its upper constant segment with $\phi$ wider than $\eta_{i}$, and then $\phi$ is redundant with $M$. This proves (i). \\
Since $C$ is finite then $\vert {{\bf{U}}^{c}}(\eta_i) \vert < \vert {{\bf{U}}^{c}}(\phi) \vert$. If any of the atoms $\eta_i$ is redundant with $M$, then $\eta_i$ can also be expressed as a union of atoms of $M$. We can continue replacing the redundant atoms for others with ever smaller upper constant segments until reaching non-redundant atoms of $M$. Because $\bigtriangledown$ is associative there is at least one decomposition of $\phi$ as a union of non-redundant atoms of $M$. 
\end{proof}
\bigskip

\begin{theorem} \label{uniqueAtomization}
i) Two atomizations of the same model have the same non-redundant atoms. \\
ii) Any model has a unique atomization without redundant atoms. 
\end{theorem}
\begin{proof}
Let $A$ and $B$ be two atomizations of a model M without redundant atoms. Choose a non-redundant atom $\phi$ in $B$ and consider the model $[A + \{\phi\}]$ spawned by $\phi$ and the atoms of $A$.  It is clear that $A + \{\phi\}$ spawns the same model than $A$, otherwise there would be a duple that is positive for $[A]$ discriminated only by $\phi$ in $[A + \{\phi\}]$ and, hence, also discriminated by $\phi$ in $[B]$, i.e. negative for $[B]$, contradicting that $A$ and $B$ are atomizations of the same model. From theorem \ref{redundantAtom} either $\phi$ is a non-redundant atom in $[A]$ or $\phi$ is redundant with $[A]$. Assume $\phi$ is redundant with $[A]$. There is a set $E_{\phi}$ of atoms of $A$ such that $\phi$ is a union of the atoms in $E_{\phi}$ (see theorem \ref{redundantIsUnionOfAtomsTheorem}). Choose an atom $\eta$ in $E_{\phi}$ and consider the model $[B + \{\eta\}]$. The same reasoning applies, so we should get that either $\eta$ is in $B$ or is redundant with atoms of $B$. We can substitute $\eta$ with the atoms that make $\eta$ redundant with $B$, and do the same for every atom of $E_{\phi}$ to form a set $E'_{\phi}$ of atoms of $B$. It follows that $\phi$ is a union of the atoms in $E'_{\phi}$ which implies that $\phi$ is redundant with $B$, against our assumptions. Therefore, $\phi$ cannot be redundant with $A$ and then $\phi$ should be a non-redundant atom in $A$, which proves proposition i, i.e. that all the non-redundant atoms of $B$ are in $A$ and vice-versa. Finally, proposition ii follows from proposition i and theorem \ref{redundantAtom}. 
 \end{proof}	
\bigskip
\bigskip

If $A$ and $B$ are models or sets of atoms, we use $A + B$ to represent the model $[A \cup B]$, a model atomized by the atoms of $A$ and $B$.

\bigskip
\begin{theorem} \label{unionOfModels}
Let $A$ and $B$ be two sets of atoms that atomize models $A$ and $B$ over $C$. \\
$Th^{+}_0(A + B) = Th^{+}_0(A) \cap Th^{+}_0(B)$ and $Th^{-}_0(A + B) = Th^{-}_0(A) \cup Th^{-}_0(B)$.
\end{theorem}
\begin{proof} 
Let $r =(a, b)$ be a duple over $C$. From the third axiom, AS3, and theorem \ref{atomicSegmentFromTermTheorem}(iv) follows that $r \in Th^{-}_0(A \cup B)$ if and only if the discriminant ${\bf{dis}}_{A + B}(r)$ is not empty. Operating on the definition of discriminant ${\bf{dis}}_{A + B}(r)= {{\bf{L}}^{a}_{A + B}}(a) - {{\bf{L}}^{a}_{A + B}}(b) = {{\bf{L}}^{a}_{A}}(a) \cup{ {\bf{L}}^{a}_{B}}(a) - {{\bf{L}}^{a}_{A}}(b) \cup{ {\bf{L}}^{a}_{B}}(b)$. Due to axiom AS5 sets $A$ and $B$ may have a non-empty intersection. However, whether an atom is in the lower segment of terms $a$ or $b$ depends only on the atom, so ${{\bf{L}}^{a}_{A}}(a) \cap { {\bf{L}}^{a}_{B}}(b) \subseteq {{\bf{L}}^{a}_{A}}(a) \cap { {\bf{L}}^{a}_{A}}(b)$ and ${{\bf{L}}^{a}_{B}}(a) \cap { {\bf{L}}^{a}_{A}}(b) \subseteq {{\bf{L}}^{a}_{B}}(a) \cap { {\bf{L}}^{a}_{B}}(b)$, and from here if follows that ${{\bf{L}}^{a}_{A}}(a) \cup{ {\bf{L}}^{a}_{B}}(a) - {{\bf{L}}^{a}_{A}}(b) \cup{ {\bf{L}}^{a}_{B}}(b) = ({{\bf{L}}^{a}_{A}}(a) - {{\bf{L}}^{a}_{A}}(b)) \cup ({ {\bf{L}}^{a}_{B}}(a) - { {\bf{L}}^{a}_{B}}(b))$ and, finally, ${\bf{dis}}_{A + B}(r) = {\bf{dis}}_A(r)  \cup  {\bf{dis}}_B(r)$. Therefore, $r$ is discriminated in $A + B$ if and only if it is discriminated either in $A$ or in $B$, and then $Th^{-}_0(A + B) = Th^{-}_0(A) \cup Th^{-}_0(B)$. Taking the complementary sets we get $Th^{+}_0(A + B) = Th^{+}_0(A) \cap Th^{+}_0(B)$ or, in other words, $r$ is not discriminated in $A + B$ if and only if $r$ is not discriminated neither in $A$ nor in $B$.
\end{proof}	
\bigskip
\bigskip

\begin{theorem} \label{unionWithFreer}
Let $A$ and $B$ be two sets of atoms that atomize models $A$ and $B$, with model $A$ freer or as free as $B$.\\
 i) The model $A + B$ spawned by the atoms of $A$ and the atoms of $B$ is the same as the model spawned by $A$ alone. \\
 ii) The atoms of $B$ are atoms of $A$ or are redundant with the atoms of $A$.  \\
 iii) $A$ is freer than $B$ if and only if the atoms of $B$ are atoms of $A$ or unions of atoms of $A$.
\end{theorem}
\begin{proof} (i) Since $A$ is freer or as free as $B$ all the negative duples of $B$ are also negative in $A$. This means that a duple discriminated by an atom of $B$ is also discriminated by some atom of $A$. In addition, each positive duple of $A$ is also a positive duple of $B$ and also a positive duple of $A + B$. Therefore, a duple of $A + B$ is  positive if and only if is positive in $A$. Likewise, a duple is negative if and only if is negative in $A$ and then models $A + B$ and $A$ are equal.   \\
(ii) follows from theorem \ref{uniqueAtomization} and the fact that $A + B$ spawns the same model than $A$. \\
(iii) Left to right. Assume $A$ is freer than $B$. Consider the model $A + B$ spawned by the atoms of $A$ and the atoms of $B$. Proposition (i) tells us that if $A$ is freer than $B$ the model $A + B$ is equal to $A$. Theorem \ref{redundantAtom} says that each atom $\phi$ of $B$ is an atom of $A$ or is redundant with the atoms of $A$. In other words, either $\phi$ is an atom of $A$ or for each constant $c$ such that $\phi < c$ there is at least one atom $\eta < c$ in $A$ such that $\phi$ is wider than $\eta$, i.e. the set  ${{\bf{U}}^{c}}(\phi)$ contains the set ${{\bf{U}}^{c}}(\eta)$. Since there is an $\eta$ for each constant $c \in {{\bf{U}}^{c}}(\phi)$, if the set $\{\eta_{i}:i=1,...,n\}$ makes $\phi$ redundant in $A$, then ${{\bf{U}}^{c}}(\phi) = \cup_{i}{{\bf{U}}^{c}}(\eta_{i})$ and $\phi$ is the union $\phi = \bigtriangledown_{i} \eta_{i}$.  \\
Right to left.  Assume now that the atoms of $B$ are atoms of $A$ or unions of atoms of $A$. If an atom $\phi$ of $B$ is a union of atoms $\{\eta_{i}:i=1,...,n\}$ of $A$ then $\phi$ is redundant with $A$ and any duple discriminated by $\phi$ is discriminated by some of the atoms $\{\eta_{i}:i=1,...,n\}$. Therefore, any duple $r$ discriminated by an atom of $B$ is discriminated by at least one atom of $A$ so, $B \models r^{-}$ implies $A \models r^{-}$ and $A$ is freer or equal to $B$.
\end{proof}	

\bigskip
\bigskip

\begin{definition} \label{fullCrossing} Let $r = (r_L,r_R)$ be a duple and $M$ an atomized semilattice model, both over the same set of constants. The full crossing of $r$ in $M$, written $\square_{r}M$, is the model atomized by $(M - H) \cup (H\bigtriangledown B)$ where $H = {\bf{dis}}_M(r)$, $B = {\bf{L}}^a_M(r_R)$ and $H\bigtriangledown B \equiv \{\lambda \bigtriangledown \rho : (\lambda \in H) \wedge (\rho \in B) \}$.   
\end{definition} 

\bigskip
\bigskip

We define $F_C(R^{+})$ as the freest model that satisfies a set of positive duples $R^{+}$. The next theorem gives a recipe to construct freest models.

\bigskip

\begin{theorem} \label{fullCrossingIsFreestTheorem}
Let $M$ be an atomized model with or without redundant atoms and $r$ a duple such that $M \models r^{-}$. The full crossing of $r$ in $M$ is the freest model $F_C(Th_0^{+}(M) \cup\, r^{+})$.  
\end{theorem}
\begin{proof}
Let $H \subseteq M$ be the discriminant of $r=(r_L,r_R)$, i.e. the set of atoms $\lambda$ such that $\lambda < r_L \,\wedge\, \lambda \not< r_R$. Let $B \subseteq M$ be the set of atoms of $r_R$, i.e. the atoms $\rho$ such that $\rho < r_R$. The full crossing of $(r_L,r_R)$ in model $M$ is the model $K = (M - H) \cup (H\bigtriangledown B)$ where $H\bigtriangledown B \equiv \{\lambda \bigtriangledown \rho : (\lambda \in H) \wedge (\rho \in B) \}$ is the set of all pairwise unions of an atom of $H$ and an atom of $B$.   \\
Using the properties of theorem \ref{unionProperties} (property iii) follows that the atoms $\lambda \bigtriangledown \rho \in H\bigtriangledown B \subseteq K$ introduced by the full crossing operation satisfy $\lambda \bigtriangledown \rho < r_R$ because $\rho < r_R$. Since the atoms in the discriminant $\lambda \in H = {\bf{dis}}_M(r)$ are no longer present in $M - H$ and the atoms introduced in $H\bigtriangledown B$ are all in the lower segment of $r_R$ then $K \models r^{+}$.  \\
If a model $N$ is atomized by a subset of the atoms of $M$ then there is no duple $r$ positive in $M$ and negative in $N$. This should be clear, from theorem \ref{unionOfModels} since $Th^{-}_0(M)  = Th^{-}_0(M + N) = Th^{-}_0(M) \cup Th^{-}_0(N)$ from which $Th^{-}_0(N) \subseteq Th^{-}_0(M)$. In other words, the elimination of atoms from a model cannot cause any positive duple to become negative. Hence, the atoms in $H$ eliminated by the full crossing operation cannot switch positive duples into negative. We have to prove that the atoms introduced by the full crossing in $H\bigtriangledown B$ do not switch positive duples into negative duples either.  Assume $M \models s^{+}$ for some duple $s=(s_L,s_R)$.  Suppose that term $s_L$ acquires one of the new atoms $\lambda \bigtriangledown \rho$ in its lower segment. From \ref{unionProperties} (property iv) follows that either $\lambda < s_L$ or  $\rho< s_L$. Because $M \models (s_L \leq s_R)$ and $\lambda$, $\rho$ are atoms of $M$ then $(\lambda < s_R) \vee (\rho < s_R)$. Using now \ref{unionProperties} (property iii) we get $\lambda \bigtriangledown \rho < s_R$. Therefore, the atoms of the form $\lambda \bigtriangledown \rho$ cannot switch a positive duple $s^{+}$ into $s^{-}$.   \\
So far we know that $K \models r^{+}$ and $(M \models s^{+}) \Rightarrow (K \models s^{+})$, so the model resulting from full crossing satisfies $K \models Th_0^{+}(M) \cup r^{+}$. To prove that $K$ is the freest model of $Th_0^{+}(M) \cup r^{+}$ we have to show that full crossing does not switch negative duples into positive either unless they are logical consequences of $Th_0^{+}(M) \cup r^{+}$.  \\
Consider a duple $s=(s_L,s_R)$ such that $K \models s^{+}$ and $M \models s^{-}$. The crossing of $r$ has switched $s$ from negative to positive. For this to occur a necessary condition is that the discriminant of $s$ should disappear from $K$ as a result of the crossing, in other words, ${\bf{dis}}_M(s) \subseteq H = {\bf{dis}}_M(r)$, so the atoms of ${\bf{dis}}_M(s)$ should all be atoms of $r_L$. This implies  $s_L \leq s_R \odot r_L$ holds in $M$.   \\
Since $M \models s^{-}$ there is at least one $\lambda^{*} \in {\bf{dis}}_M(s)$. From \ref{unionProperties}, property iii, all the atoms of the form $\{  \lambda^{*} \bigtriangledown \rho :  \rho \in B \}$ are in $s_L$. If $K \models s^{+}$ then all the atoms $\lambda^{*} \bigtriangledown \rho$ should also be atoms of $s_R$ in model $K$. Using \ref{unionProperties} (property v) $\lambda^{*} \not< s_R$ and $\lambda^{*} \bigtriangledown \rho < s_R$ implies $\rho < s_R$ for each $\rho \in B$. Since $B$ is the lower atomic segment of $r_R$ in $M$ immediately follows that $M \models (r_R \leq s_R)$.  Putting both conditions together:
\[
M \models  ((s_L \leq s_R \odot r_L) \,\wedge\, (r_R \leq s_R)).
\]
Since
\[
(s_L \leq s_R \odot r_L) \,\wedge\, (r_R \leq s_R)  \,\wedge\, (r_L \leq  r_R) \Rightarrow s_L \leq  s_R
\]
$s^{+}$ is a logical consequence of the theory $Th_0^{+}(M) \cup r^{+}$ so any model of $Th_0^{+}(M) \cup r^{+}$ must satisfy $s^{+}$.  We have proved that any duple $s^{+}$ satisfied by $K$ is satisfied by any model of $Th_0^{+}(M) \cup r^{+}$ and, because $K \models Th_0^{+}(M) \cup r^{+}$ then $K = F_C(Th_0^{+}(M) \cup r^{+})$.     \\
Finally, we should make sure that $K$ satisfies the sixth axiom of atomized semilattices. This reduces to check that by removing the atoms in $H$ from $M$ we are not leaving some constant $c$ with an empty lower atomic segment. Because for each $\lambda < c$ we have at least one $\lambda \bigtriangledown \rho < c$ in $K$, it follows that $K$ satisfies the sixth axiom if $M$ does.
\end{proof}
\bigskip

We use $\square_{a \leq b} M$ or $\square_{r} M$ for a duple $r=(a,b)$ to represent the model resulting from the full crossing of $r$ in M.

\bigskip

\begin{example} \label{simplecrossingExample} Suppose the model over the constants $C = \{ a,b,c,d,e \}$:  \[
M = [\phi_{a},\, \phi_{a,b}, \phi_{c,d,e},\, \phi_{b,e},\, \phi_{c}, \, \phi_{d} ].
\]
where we are using $\phi_A$ for an atom with upper constant segment equal to the set $A$. \\
Note that $M \models (b \not\leq a \odot d)$ as the duple $(
b,\,  a \odot d)$ is discriminated by $\,{\bf{dis}}_M(b, \,a \odot d) =  \{ \phi_{b,e} \}$. The full crossing of $(
b, \, a \odot d)$ on $M$ is equal to: \[
\square_{b \leq a \odot d} M = [ (M - \phi_{b,e}) \cup \phi_{b,e} \bigtriangledown \{\phi_{a}, \phi_{a,b}, \phi_{c,d,e}, \, \phi_{d}\} ] = 
\] \[
= [\{ \phi_{a},\, \phi_{a,b}, \phi_{c,d,e},\, \phi_{c}, \, \phi_{d} \} \cup \{ \phi_{a,b,e}, \phi_{a,b,e},\, \phi_{b,c,d,e}, \, \phi_{b,d,e}  \} ] = 
\]\[
= [ \phi_{a},\, \phi_{a,b}, \phi_{c,d,e},\, \phi_{c}, \, \phi_{d}, \, \phi_{a,b,e}, \phi_{b,c,d,e}, \, \phi_{b,d,e} ]. 
\]
Since $\phi_{b,c,d,e} = \phi_{c,d,e} \bigtriangledown \phi_{b,d,e}$ then $\phi_{b,c,d,e}$ is redundant with $\square_{b \leq a \odot d} M$, and then:\[
\square_{b \leq a \odot d} M = [ \phi_{a},\, \phi_{a,b}, \phi_{c,d,e},\, \phi_{c}, \, \phi_{d}, \, \phi_{a,b,e}, \, \phi_{b,d,e} ],  
\]
that is the freest model that satisfies $\,\square_{b \leq a \odot d} M \models (b \leq a \odot d) \cup Th_0^{+}(M)$.
\end{example}

\bigskip
\bigskip

\begin{theorem} \label{fullCrossingIsCommutative}
The full crossing operation is commutative up to redundant atoms. The full-crossing of every duple in an ordered set of duples $R$ produces the same model for every order chosen. 
\end{theorem}
\begin{proof}
Theorem \ref{fullCrossingIsFreestTheorem} shows that the result of full crossing the duples of $R^{+}$ in a model $M$ is always the freest model $F_C(Th_0^{+}(M) \cup R^{+})$ and, hence, is independent of the order of crossing. Because the result is the same model, from theorem \ref{redundantAtom}, the resulting atomizations for different orders of crossing can only differ in redundant atoms. 
\end{proof}
\bigskip
\bigskip

\begin{theorem} \label{freestModelTheorem}
The freest model $F_C(\emptyset)$ over a set $C$ of constants has $\vert C \vert$ non-redundant atoms, each with a single constant in its upper segment. 
\end{theorem}
\begin{proof}
The freest model is the term algebra, i.e. the model spawned by the terms modulus the rules of the algebra, in this case the commutative, associative and idempotent laws. For terms $s$ and $t$, the term algebra satisfies that $F_C(\emptyset) \models s \leq t$ if and only if the component constants satisfy ${\bf{C}}(s) \subseteq {\bf{C}}(t)$. Let $M$ be the atomized model obtained by $\vert C \vert$ different atoms, each with one constant in its upper segment. From the axiom of atomized models $\,\phi < a \odot b \, \Leftrightarrow  (\phi < a) \vee (\phi < b)\,$ applied to the component constants of $s$ and $t$ follows that for $M$ (and for any semilattice), ${\bf{C}}(s) \subseteq {\bf{C}}(t)$ implies $s \leq t$. Conversely, assume $M \models s \leq t$. Each atom $\phi$ in the lower atomic segment of $s$ should be in the lower segment of some component constant $c \in {\bf{C}}(s)$ (this is proposition (ii) of theorem \ref{atomicSegmentFromTermTheorem}) and, since the atoms of $M$ have only one constant in its upper segment then $\phi < t$ can occur only if $c$ is also a component constant of $t$. Since each constant of $C$ has its own atom, every component constant of $s$ should be a component constant of $t$ otherwise there is an atom that discriminates $(s, t)$ against our assumption. Since each atom $\phi$ of $M$ has an upper segment ${{\bf{U}}^{c}}(\phi)$ with a single constant, $\phi$ is non-redundant. It follows that $M$ is an atomized semilattice model without redundant atoms that satisfies the same positive an negative duples that $F_C(\emptyset)$ and, hence, $M$ is equal to $F_C(\emptyset)$.
\end{proof}
\bigskip

We can easily show that the model $M$ of theorem\ref{freestModelTheorem}, i.e. the model spawned by $\vert C \vert$ different atoms each with one constant in its upper segment, is freer than any atomized model. Let $N$ be any atomized model and let $r$ be any duple such that $N \models r^{-}$. Each atom $\eta$ of $N$ is the union of the atoms of $M$ corresponding to each constant in the set ${{\bf{U}}^{c}}(\eta)$ and this proves that $\eta$ is redundant with $M$. Since all the atoms of $N$ are redundant with $M$, using theorem $\ref{unionWithFreer}$ proposition (iii) we conclude that $M$ is freer or as free as $N$ and, therefore, $M \models r^{-}$ and $M$ is the freest model. From theorem \ref{redundantAtom} any atomization of $F_C(\emptyset)$ contains the $\vert C \vert$ atoms of $M$ and only can differ from $M$ in a set of redundant atoms. 

\bigskip
\begin{theorem} \label{atomizationExistsTheorem}
Any model $M$ with a finite set $C$ of constants can be atomized. 
\end{theorem}
\begin{proof}
Since the theory $Th_0^{+}(M)$ for any model $M$ over a finite set of constants is a finite a set, $M$ can be atomized by starting with the set of $\vert C\vert$ atoms that provides an atomization for the freest model $F_C(\emptyset)$ (see theorem \ref{freestModelTheorem}), each atom contained in a single constant, and then by performing a finite sequence of full crossing operations for each duple in $Th_0^{+}(M)$. As a result we obtain an atomization of model $M$ which follows from theorem \ref{fullCrossingIsFreestTheorem} and $M = F_C(Th_0^{+}(M))$.   
\end{proof}
\bigskip

The reader may recognize theorem \ref{atomizationExistsTheorem} as a consequence of the Stone's theorem  \cite{Burris}. In theorem \ref{modelAsASubdirectProduct} we will make the connection explicit.

\bigskip
\bigskip
\begin{theorem} \label{fullCrossingAndRedundantAtomsTheorem}
The models resulting from the full crossing of a duple $r$ in different atomizations of a model $M$ can only differ in redundant atoms.
\end{theorem}
\begin{proof}
The proof of theorem \ref{fullCrossingIsFreestTheorem} needs no assumption regarding redundant atoms. With or without redundant atoms in the atomization of a model $M$ the result of full crossing a duple $r$ in $M$ is the freest model $F_C(Th_0^{+}(M) \cup r^{+})$. Theorem \ref{redundantAtom} assures us that if the resulting models are the same then the atomizations can only differ in redundant atoms. 
\end{proof}
\bigskip
\bigskip

\begin{theorem} \label{crossingProduceLessFreeModels}
Let $M$ be an atomized semilattice. If $M \models r^{-}$ then $M$ is freer than $\square_r M$. If $M \models r^{+}$ then $M = \square_r M$.
\end{theorem}
\begin{proof}
Theorem \ref{fullCrossingIsFreestTheorem} states that $\square_r M$ is the freest model $F_C(Th_0^{+}(M) \cup \{ r^{+} \})$. Since  $\square_r M \models Th_0^{+}(M) \cup \{ r^{+} \}$ then $Th_0^{+}(M) \subset Th_0^{+}(\square_r M)$. Since $\square_r M$ and $M$ are generated by the same constants, $Th_0^{+}(M) \subset Th_0^{+}(\square_r M)$ is equivalent to $Th_0^{-}(\square_r M) \subset Th_0^{-}(M)$, and it follows that $M$ is freer than $\square_r M$. \\
If $M \models r^{+}$ the discriminant of $r$ in $M$ is empty and full crossing leaves $M$ unaltered. Alternatively, $\square_r M = F_C(Th_0^{+}(M) \cup \{ r^{+} \}) = F_C(Th_0^{+}(M)) = M$.
\end{proof}
\bigskip

\newpage

\section{Join models, restrictions, subalgebras and products } \label{unionJoinRestrictionSubalgebraAndProducts}

We use $M^{|Q}$ or $[|Q]M$ to refer to the ``restriction of $M$ to $Q$'', a model defined as the subalgebra of $M$ generated by $Q$. Alternatively we can define $M^{|Q}$ as the smallest subalgebra of $M$ that contains the constants $Q$. 

\bigskip
\begin{definition} \label{restrictionDefinition}  
Let $Q$ be a subset of the constants $Q \subset C$. The ``restriction of $\phi$ to $Q$'', is the atom $\phi^{|Q}$ with upper segment: $U^{c}(\phi^{|Q}) = U^{c}(\phi)\cap Q$. The restriction of $\phi$ to $Q$ does not exist if $U^{c}(\phi)\cap Q = \emptyset$.
\end{definition} 
\bigskip

\begin{theorem}  \label{restrictionLemma} 
Let $M$ be an atomized semilattice model over a set of constants $C$ and  $Q$ be a subset of $C$. Let $M^{|Q}$ be the subalgebra of $M$ generated by $Q$. Then $M^{|Q}$ is the model spawned by the restriction to $Q$ of the atoms of $M$. The restriction of the atoms of a model $M$ to a subset $Q$ of the constants produces the same model $M^{|Q}$ for every atomization of $M$. 
\end{theorem}
\begin{proof}
Let $r \equiv (a, b)$ be a duple and $a$ and $b$ terms over $Q$. Let $\phi$ be an atom of $M$ and $\phi^{|Q}$ its restriction to $Q$ defined by the upper constant segment $U^{c}(\phi^{|Q}) = U^{c}(\phi)\cap Q$. Let ${\bf{C}}(a)$ be the component constants of $a$. By definition, $(\phi^{|Q} < a)$ holds when $(U^{c}(\phi)\cap Q) \cap {\bf{C}}(a) \not= \emptyset$ and, since $a$ is a term over $Q$, then $U^{c}(\phi) \cap {\bf{C}}(a) \not= \emptyset$. It follows that $(\phi^{|Q} < a)$ if and only if $(\phi < a)$. Hence, if $r$ is discriminated by an atom of $M$ then $r$ is discriminated by the restriction to $Q$ of the same atom. Additionally, if $r$ is discriminated by no atom in $M$ (i.e. $M \models a \leq b$) then no restriction to $Q$ of any atom of $M$ can discriminate $r$. Therefore, the model spawned by the restriction of the atoms is the subalgebra $M^{|Q}$ regardless of the atomization chosen for $M$.
\end{proof}

\bigskip
\bigskip

We have seen that the restriction of the atoms of $M$ to $Q$ spawns the well-defined semilattice $M^{|Q}$. Theorem \ref{uniqueAtomization} tells us that the non-redundant atoms of a model are unique, so each non-redundant atom of $M^{|Q}$ has to be equal to the restriction to $Q$ of at least one atom of $M$. In addition, we could choose for $M$ an atomization containing only non-redundant atoms in $M$ which implies that each non-redundant atom of $M^{|Q}$ is the restriction of at least one non-redundant atom in $M$. We provide here another more direct proof that the reader may find illuminating:

\bigskip
\bigskip

\begin{theorem}  \label{restrictionAndRedundacyLemma} Let $M$ be an atomized semilattice model over $C$ and let $Q \subset C$. For each non-redundant atom $\alpha$ of $M^{|Q}$ there is at least one non-redundant atom of $M$ that restricted to $Q$ is equal to $\alpha$.
\end{theorem}
\begin{proof}
Theorem \ref{restrictionLemma} says that each $\alpha$ in $M^{|Q}$ is the restriction of some atom of $\phi \in M$. i.e. $\alpha = \phi^{|Q}$. Either $\phi$ is a non-redundant atom of $M$ or $\phi$ is redundant with $M$ and then $\phi = \bigtriangledown^n_i \varphi_i$ where $\{\varphi_1, \varphi_2,...,\varphi_n\}$ are all non-redundant atoms in $M$. From $U^{c}(\phi) = \cup_i U^{c}(\varphi_i)$ we get  $U^{c}(\phi) \cap Q = \cup_i \, (U^{c}(\varphi_i) \cap Q)$ so we can write $\phi^{|Q} = \bigtriangledown^m_k \varphi_k^{|Q}$ with $m \leq n$ where the summation with index $k$ runs along the $\varphi^{|Q}$ for which $U^{c}(\varphi_k) \cap Q \not= \emptyset$, in other words, it runs along the $\varphi^{|Q}$ that exist.  Again, theorem \ref{restrictionLemma} tells us that $\varphi_k^{|Q} \in M^{|Q}$. If $\alpha$ is non-redundant in $M^{|Q}$, since $\alpha = \phi^{|Q} = \bigtriangledown^m_k \varphi_k^{|Q}$, then $\alpha$ can be equal to a union of atoms of $M^{|Q}$ only if at least one of the atoms in the union is equal to $\alpha$ (otherwise $\alpha$ would be redundant with $M^{|Q}$), i.e. $\alpha = \phi^{|Q} = \varphi_k^{|Q}$ for at least one value of $k$. It follows that if $\alpha$ is non-redundant in $M^{|Q}$ then $\alpha$ is equal to the restriction of at least one non-redundant atom in $M$. 
\end{proof}

\bigskip
\bigskip

\begin{theorem}  \label{homomorphismIfZeroAtom} Let $M$ be an atomized semilattice model over $C$ and let $Q \subset C$. There is a homomorphism from $M$ onto $M^{|Q}$ that maps each term over $C$ to its restriction to $Q$ if and only if $\ominus_Q$ is a non-redundant atom of $M^{|Q}$.
\end{theorem}
\begin{proof}
Assume $\ominus_Q$ is a non-redundant atom of $M^{|Q}$. Theorem \ref{ceroTheorem} says that there is at least one constant $k \in Q$ such that $M^{|Q} \models \forall x (k \leq x)$.  Let $h$ map each constant of $C - Q$ to $k$ and each constant in $Q$ to the constant with the same name in $M^{|Q}$. Each term $t$ over $C$ is either a term over $Q$ or a term over $C - Q$ or an idempotent summation of a term $r$ over $Q$ and a term $s$ over $C - Q$. Since $h(r \odot s) =  h(r) \odot h(s) =  r \odot k = r$, then $h$ is a homomorphism from $M$ onto its subalgebra $M^{|Q}$ that maps each term over $C$ to its restriction to $Q$. \\ 
On the the other direction, assume that there is a homomorphism $h$ from $M$ onto its subalgebra $M^{|Q}$ that maps each term over $C$ to its restriction to $Q$. Choose a term $s$ over $C - Q$. For every term $r$ over $Q$ we have $h(r) \odot h(s) =  h(r \odot s)$ and, since $h$ maps $r \odot s$ to its restriction to $Q$ then $h(r \odot s) = h(r) = r$. It follows that $h(s)$ behaves as neutral element, i.e. $h(s) < x$ for every $x$ in $M^{|Q}$ and, again from theorem \ref{ceroTheorem}, a neutral element in $M^{|Q}$ implies $\ominus_Q$ is a non-redundant atom of $M^{|Q}$.
\end{proof}

\bigskip

Note that it is is always possible to add a new constant to $C$ and $Q$ that behaves as a neutral element and then build a homomorphism from $M$ onto $M^{|Q}$ that maps each term over $C$ to its restriction to $Q$.

\bigskip
\bigskip

\begin{theorem}  \label{restrictionAndCrossingCommute}  Let $r$ be a duple over $Q \subset C$ and $M$ a model over $C$: \\
i) ${\bf{dis}}_{[|Q] M}(r)  = [|Q] {\bf{dis}}_{M}(r)$, \\
ii) Restriction and crossing commute: $\square_{r} [|Q] M = [|Q] \square_{r} M$. 
\end{theorem}
\begin{proof}
(i) Let $r = (r_L, r_R)$.  The discriminant of $r$ in $[|Q] M$ is: \[
{\bf{dis}}_{[|Q] M}(r)  = \{ \phi : (\phi \in [|Q] M) \wedge (\phi < r_L) \wedge (\phi  \not< r_R) \} =   \{ \phi^{|Q} : (\phi \in M) \wedge (\phi^{|Q} < r_L) \wedge (\phi^{|Q}  \not< r_R) \},
\]
where we have used theorem \ref{restrictionAndRedundacyLemma}. Since $r$ is a duple over $Q$ then: \[
 \{ \phi^{|Q} : (\phi \in M) \wedge (\phi^{|Q} < r_L) \wedge (\phi^{|Q}  \not< r_R) \} = \{ \phi^{|Q} : (\phi \in M) \wedge (\phi < r_L) \wedge (\phi \not< r_R) \} = |Q] {\bf{dis}}_{M}(r).
\]
(ii) From the definition of full crossing: \[
\square_{r} [|Q] M =  ([|Q] M - {\bf{dis}}_{[|Q] M}(r)) +  {\bf{dis}}_{[|Q] M}(r) \bigtriangledown {\bf{L}}^{a}_{[|Q] M}(r_R) = 
\]\[
=  ([|Q] M - [|Q] {\bf{dis}}_{M}(r)) + ([|Q] {\bf{dis}}_{M}(r)) \bigtriangledown  ([|Q] {\bf{L}}^{a}_{M}(r_R)).
\]
where we have used proposition (i) and taken into account that $r_R$ is a term over $Q$.  \\
Consider the sets of atoms $A = M - {\bf{dis}}_{M}(r)$ and $B = [|Q] M - [|Q] {\bf{dis}}_{M}(r)$. We are going to show now that $[|Q] A = B$. Suppose there is an atom $\alpha \in [|Q] A$ such that $\alpha \not\in B$. Because $\alpha \in [|Q] A$ implies $\alpha \in [|Q] M$ then $\alpha \not\in B$ requires $\alpha \in [|Q] {\bf{dis}}_{M}(r)$. Since $r$ is a duple over $Q$, for every $\alpha \in [|Q] {\bf{dis}}_{M}(r)$ we have that $\forall \phi ( (\alpha = \phi^{|Q}) \Rightarrow  (\phi \in {\bf{dis}}_{M}(r)))$ which implies $\forall \phi ( (\alpha = \phi^{|Q}) \Rightarrow  (\phi \not\in A))$. Since theorem \ref{restrictionAndRedundacyLemma} says that such $\phi$ should exist in $M$, then $\alpha \not\in [|Q] A$, contradicting our assumption. Hence: $[|Q] A \subseteq B$.  \\
Suppose now there is an atom $\beta \in B$ such that  $\beta \not\in [|Q] A$. Like before, $\beta \in B$ implies $\beta \in [|Q] M$ and then $\beta \not\in [|Q] A$ requires $\forall \phi ( (\beta = \phi^{|Q}) \Rightarrow  (\phi \in {\bf{dis}}_{M}(r)))$, which in turn, implies $\beta \in [|Q] {\bf{dis}}_{M}(r)$ and then $\beta \not\in B$; a contradiction. It follows $B \subseteq [|Q] A$ and considering the inclusion in the other direction  $[|Q] A = B$. \\
Since every atom in the discriminant ${\bf{dis}}_{M}(r)$ and every atom in ${\bf{L}}^{a}_{M}(r_R)$ have a restriction to $Q$ then $([|Q] {\bf{dis}}_{M}(r)) \bigtriangledown  ([|Q] {\bf{L}}^{a}_{M}(r_R)) = [|Q]  ({\bf{dis}}_{M}(r)\bigtriangledown {\bf{L}}^{a}_{M}(r_R))$. Substituting above and also replacing $B$ by $[|Q] A$ we get $([|Q] M - [|Q] {\bf{dis}}_{M}(r)) + ([|Q] {\bf{dis}}_{M}(r)) \bigtriangledown  ([|Q] {\bf{L}}^{a}_{M}(r_R)) = [|Q] \square_{r} M$, which proves the theorem.
\end{proof}

\bigskip

\begin{theorem}  \label{freedomCommutesWithCrossing} Let $C$ be a set of constants and $Q \subseteq C$. Let $M$ be a semilattice model over $C$ and $N$ a semilattice model over $Q$. Let $r$ be a duple over $Q$. If $M$ is as free or freer than $N$ then $\square_r M$ is as free or freer than $\square_r N$.
\end{theorem}
\begin{proof}
$M$ is freer or as free as $N$, therefore $Th_0^{-}(N) \subseteq Th_0^{-}(M)$. Assume first that $Q = C$. \\
When $Q = C$, it is also true that $Th_0^{+}(M) \subseteq Th_0^{+}(N)$. Suppose there is a duple $s \in Th_0^{-}(N)$ for which $\,\square_r N \models s^{-}$ and $\square_r M \models s^{+}$. Consider the union model $\square_r N + \square_r M$ spawned by the atoms of both models. From theorem \ref{unionOfModels} follows that $\square_r N + \square_r M \models  Th_0^{+}(M) \cup \{ r^{+}\} \cup \{ s^{-}\}$. Theorem \ref{fullCrossingIsFreestTheorem} says that  $\square_r M$ is the freest model that satisfies $Th_0^{+}(M) \cup \{ r^{+}\}$, i.e. $\square_r M =  F_C(Th_0^{+}(M) \cup \{ r^{+}\})$ and we have assumed $\square_r M \models s^{+}$. However, model $\square_r N + \square_r M$ is a model of $Th_0^{+}(M) \cup \{ r^{+}\}$ that do not satisfy $s^{+}$ which contradicts that the freest model satisfies $s^{+}$. Therefore, such $s$ does not exist and then $Th_0^{-}(\square_r N) \subseteq Th_0^{-}(\square_r M)$. \\
Consider now that $Q \subset C$. All the duples in $Th_0^{-}(N)$ are duples over $Q$. Since every positive or negative duple over $Q$ satisfied by $M$ is also satisfied by its subalgebra $[|Q]M$ then $Th_0^{-}(N) \subseteq Th_0^{-}([|Q]M) = Th_0^{-}(M) \cap (F_Q(\emptyset) \times F_Q(\emptyset))$. Since $N$ and $[|Q]M$ are models over the same constants, $Q$, we have just shown that $Th_0^{-}(N) \subseteq Th_0^{-}([|Q]M)$ implies that $Th_0^{-}(\square_r N) \subseteq Th_0^{-}(\square_r [|Q]M)$ and, using theorem \ref{restrictionAndCrossingCommute} which says that crossing and restriction commute, $Th_0^{-}(\square_r N) \subseteq Th_0^{-}( [|Q]\square_rM) \subseteq Th_0^{-}(\square_rM)$ and then $\square_r M$ is as free or freer than $\square_r N$.
\end{proof}

\bigskip

\begin{theorem}  \label{construibleSubmodelTheorem}  Let $C$ be a set of constants and $Q \subseteq C$. Let $M$ be a semilattice model over $C$ and $N$ a semilattice model over $Q$. The following statements are equivalent: \\
i)  $M$ is as free or freer than $N$, \\
ii) There is a set of duples $\{r_1, r_2,...,r_v\}$ over $C$ such that $N + [\ominus_c] = \square_{r_1} \square_{r_2}... \square_{r_v} M$, \\
iii) There is a set of duples $\{r_1, r_2,...,r_u\}$ over $Q$ such that $N = [|Q] \square_{r_1} \square_{r_2}... \square_{r_u} M$, \\
iv) $N$ is a subset model of $M$, written $N \subseteq M$, i.e. the atoms of $N$ are all atoms of $M$ or redundant with $M$.
\end{theorem}
\begin{proof}
$(i) \Rightarrow (ii)$ Proposition (i) assumes $Th_0^{-}(N) \subseteq Th_0^{-}(M)$, and since $\ominus_c$ discriminates no duple $Th_0^{-}(N + [\ominus_c]) = Th_0^{-}(N)$ and then $Th_0^{-}(N + [\ominus_c]) \subseteq Th_0^{-}(M)$. Because $N + [\ominus_c]$ is a model over $C$, it holds that $Th_0^{+}(M) \subseteq Th_0^{+}(N + [\ominus_c])$. Consider the set of duples $R =  \{r_1, r_2,...,r_v\}  = Th_0^{-}(M) - Th_0^{-}(N + [\ominus_c]) = Th_0^{+}(N + [\ominus_c]) - Th_0^{+}(M)$. It follows that  $\square_{r_1} \square_{r_2}... \square_{r_v} M$ produces the freest model of $Th_0^{+}(M) \cup R^{+} = Th_0^{+}(N + [\ominus_c])$, a model that is equal $N + [\ominus_c]$.  \\
$(ii) \Rightarrow (i)$ According to theorem \ref{crossingProduceLessFreeModels}, full crossing either leaves the model unaltered or produces strictly less free models. Therefore, $M$ is as free or freer than $N + [\ominus_c]$, in other words, $Th_0^{-}(N + [\ominus_c]) \subseteq Th_0^{-}(M)$ and, since $Th_0^{-}(N + [\ominus_c]) = Th_0^{-}(N)$, then $Th_0^{-}(N) \subseteq Th_0^{-}(M)$ and $M$ is as free or freer than $N$.  \\
$(i) \Rightarrow (iii)$   Proposition (i) assumes $Th_0^{-}(N) \subseteq Th_0^{-}(M)$. Since $Th_0^{-}(N)$ is a set of duples over $Q$ then $Th_0^{-}(N) \subseteq Th_0^{-}(M) \cap (F_Q(\emptyset) \times F_Q(\emptyset)) = Th_0^{-}([|Q]M)$.  Consider the set of duples $R =  \{r_1, r_2,...,r_u\}  = Th_0^{-}([|Q]M) - Th_0^{-}(N) = Th_0^{+}(N) - Th_0^{+}([|Q]M)$. Note that $N \models r^{+}$ for every $r \in R$. Since both are models over $Q$ then $Th_0^{+}([|Q]M) \subseteq Th_0^{+}(N)$ and it follows that $\square_{r_1} \square_{r_2}... \square_{r_u} [|Q]M$ produces the freest model of $Th_0^{+}([|Q]M) \cup R^{+} = Th_0^{+}(N)$, a model that is equal to $N$. Theorem \ref{restrictionAndCrossingCommute}  says that full crossing and restriction commute so $[|Q] \square_{r_1} \square_{r_2}... \square_{r_u} M = N$. \\
$(iii) \Rightarrow (i)$ Theorem \ref{restrictionAndCrossingCommute} implies $N = [|Q] \square_{r_1} \square_{r_2}... \square_{r_u} M =\square_{r_1} \square_{r_2}... \square_{r_u}  [|Q]M$ and, then, theorem \ref{crossingProduceLessFreeModels} says that $[|Q]M$ is as free or freer than $N$, i.e, $Th_0^{-}(N) \subseteq Th_0^{-}([|Q]M)$. Since $Th_0^{-}(N)$ is a set of duples over $Q$ and every positive or negative duple over $Q$ satisfied by $M$ is also satisfied by its subalgebra $[|Q]M$ then $Th_0^{-}(N) \subseteq Th_0^{-}([|Q]M) = Th_0^{-}(M) \cap (F_Q(\emptyset) \times F_Q(\emptyset))$. It follows $Th_0^{-}(N) \subseteq Th_0^{-}(M)$, and $M$ is as free or freer than $N$.  \\
$(ii) \Rightarrow (iv)$ Since $N + [\ominus_c] = \square_{r_1} \square_{r_2}... \square_{r_v} M$ then the atoms of $N$ are the result of successive crossing operations and then they are all either atoms of $M$ or union of atoms of $M$, hence, redundant with $M$.  \\
$(iv) \Rightarrow (i)$  Theorem \ref{redundantAtom} tells us that atoms redundant with $M$ only discriminate duples that are discriminated by atoms of $M$, therefore $Th_0^{-}(N) \subseteq Th_0^{-}(M)$. 
\end{proof}

\bigskip
\bigskip

\begin{theorem}  \label{homomorphismAndFreedomTheorem}  
i)  Let $M$ and $N$ be two semilattice models over $C$. If $M$ is as free or freer than $N$ then there is a homomorphism from $M$ onto $N$. \\
ii) Let $C$ be a set of constants and $Q \subseteq C$. Let $M$ be a semilattice model over $C$ and $N$ a semilattice model over $Q$. If $M$ is as free or freer than $N$ there is a homomorphism from $M^{|Q}$ onto $N$.
\end{theorem}
\begin{proof}
(i) $M$ is as free or freer than $N$, so $Th_0^{-}(N) \subseteq Th_0^{-}(M)$ and since $M$ and $N$ are both generated by the same constants, $Th_0^{+}(M) \subseteq Th_0^{+}(N)$. Let $t$ and $s$ be terms over $C$. Suppose $N \models (t \leq s)$ and $M \models (t \not\leq s)$, this is equivalent to $N \models (t = s \odot t)$ and $M \models (t \not= s \odot t)$, i.e. each difference between $Th_0^{+}(M)$ and $Th_0^{+}(N)$ reduces to some equality satisfied by $N$ and not satisfied by $M$. Consider the partition of the universe of $M$ that puts elements $a$ and $b$ of $M$ in the same class if and only if $N \models (a = b)$. Since $Th_0^{+}(M)$ is a subset of $Th_0^{+}(N)$ the partition is isomorphic to $N$ and a map from each element of $M$ to its partition class is a homomorphism from $M$ onto $N$.   \\
(ii) In the proof of theorem \ref{construibleSubmodelTheorem} we argued that if $M$ is as free or freer than $N$ then $M^{|Q}$ is also as free or freer than $N$ with both, $M^{|Q}$ and $N$, over $Q$ so from part (i) there is a homomorphism from $M^{|Q}$ onto $N$.
\end{proof}

\bigskip
\bigskip

A congruence \cite{Burris, Papert} $\theta$ of $M$ is an equivalence relation of $M$ that commutes with the operations of the algebra. In the particular case of semilattices is an equivalence relation that commutes with the idempotent operator.  $\theta$ is a congruence if:   \[
a_1\theta b_1 \, \wedge \, a_2\theta b_2  \Rightarrow (a_1\odot a_2)\theta (b_1\odot b_2).
\]
Let $a$ and $b$ be two elements of $M$. Congruences can be built from principal congruences $\Theta(a, b)$ defined as the smallest congruence with $a$ and $b$ in the same equivalence class. It is well known \cite{Burris} that any congruence $\theta =  \cup \{ \Theta(a, b) : (a, b) \in  \theta \} =  \vee \{ \Theta(a, b) : (a, b) \in  \theta \}$ where $\cup$ is the union of equivalence classes and $\vee$ is the join operator in the lattice of equivalence classes. 

The next theorem shows that any congruence on a finite semilattice can be constructed as a sequence of full crossings. 

\bigskip
 
\begin{theorem}  \label{congruenceFromFullCrossing} 
Let $\Theta(a, b)$ be a principal congruence of an atomized semilattice $M$ over $C$ with $a, b$ elements of $M$: \\
i) $\square_{a \leq  b} M \approx M/\Theta(b, b \odot a)$,\\
ii) $M/\Theta(a, b) \approx \square_{a \leq  b} \square_{b \leq  a} M$.  
\end{theorem}
\begin{proof}
Theorem \ref{fullCrossingIsFreestTheorem} says that $\square_{a \leq b} M$ is the freest model $F_C(Th_0^{+}(M) \cup\, \{ (a \leq b) \})$. Since $M$ is freer than $\square_{a \leq  b} M$ theorem \ref{homomorphismAndFreedomTheorem} says that we can build a homomorphism $h$ from $M$ to $N = \square_{a \leq  b} M$ mapping $h: \nu_{M}(t) \mapsto  \nu_{N}(t)$ where $\nu_{M}$ and $\nu_{N}$ are the natural homomorphisms and $t$ is a term over $C$. It follows from the Homomorphism Theorem \cite{Burris} of Universal Algebra that N is the quotient algebra $\square_{a \leq  b} M = M / ker(h)$ where $ker(h)$ is the kernel of the homomorphism.  \\
(i) It is well known \cite{Burris} that $ker(h)$ is also a congruence on $M$ and, since $\square_{a \leq  b} M \models (a \leq  b)$ and $(a \leq  b)  \Leftrightarrow (b = b \odot a)$ then $(b, b \odot a) \in ker(h)$. The principal congruence is the smallest congruence $\Theta(b, b \odot a)$ that has $(b, b \odot a)$, then $\Theta(b, b \odot a) \subseteq ker(h)$. On the other hand, $M/\Theta(b, b \odot a) \models  Th_0^{+}(M) \cup\, \{ (b = b \odot a) \} $. Suppose $\Theta(b, b \odot a)$ is strictly smaller than $ker(h)$, then $M/\Theta(b, b \odot a)$ is strictly freer than $\square_{a < b} M$, which contradicts theorem \ref{fullCrossingIsFreestTheorem}. Therefore,   $\Theta(b, b \odot a) = ker(h)$ and then $\square_{a \leq  b} M \approx M/\Theta(b, b \odot a)$. \\
(ii) Since $\Theta(a, b) = \Theta(b, b \odot a) \vee \Theta(a, b \odot a)$ where $\vee$ is the join operation in the lattice of equivalence classes, then $M/\Theta(a, b) \approx (M/\Theta(b, b \odot a)) /\Theta(a, b \odot a)$ and applying (i) we can write $M/\Theta(a, b) \approx \square_{a \leq  b} \square_{b \leq  a} M$.
\end{proof}

\bigskip
\bigskip

Let $M$ and $N$ be two semilattice models over the constants $C_M = \{a_1, a_2,...,a_m \}$ and $C_N  = \{b_1, b_2,...,b_n \}$ respectively, with $C_M \cap C_N = \emptyset$. We have assumed $M$ and $N$ are both generated by their constants. Consider the freest algebra $M \oplus N$ generated by the constants $C_M \cup C_N$ such that the homomorphisms $h_M$ and $h_N$ from $M$ and from $N$ to $M \oplus N$ respectively defined by $h_M(a_i) = a_i$ and $h_N(b_j) = b_j$ are one-to-one, i.e. they are embeddings. We refer to $M \oplus N$ as the ``join model'' of $M$ and $N$. 

Since we use the same names for the constants in $M$ and in $M \oplus N$ then $h_M$ is the identity and so it is $h_N$. The join model $M \oplus N$ is the closure by idempotent summations of a set formed by the elements of $M$ and the elements of $N$. Since the operator $\odot$ is commutative and associative the elements of $M \oplus N$ have the forms: $x$ or $y$ or $x \odot y$, where $x$ is an element of $M$ and $y$ an element of $N$.  

The the natural homomorphism $\nu$ from $F_{C_M \cup C_N}(\emptyset)$ onto $M \oplus N$ given by:
 \[
\nu(t) = \nu_M(t),  \,\,\,\,\,\,\, \nu(s) = \nu_N(s),  \,\,\,\,\,\,\, \nu(t \odot s) = \nu_M(t) \odot \nu_N(s).
 \]
 where $\nu_M$ and $\nu_N$ are the natural homomorphism onto $M$ and $N$ respectively and $t$ and $s$ are terms of $F_{C_M}(\emptyset)$ and $F_{C_N}(\emptyset)$ respectively.

\bigskip
\bigskip
 
\begin{theorem}  \label{sumOfModelsWhenNoIntersection} 
Let $M$ and $N$ be two atomized semilattice models over the constants $C_M$ and $C_N$ respectively, such $C_M \cap C_N = \emptyset$ and $M = [A]$ and $N = [B]$ where $A$ and $B$ are sets of atoms. The model $M \oplus N$ is atomized by $A \cup B$, i.e. $M \oplus N  \approx  M + N$.
\end{theorem}
\begin{proof}
When  $C_M \cap C_N = \emptyset$ the join model satisfies $M \oplus N  \models (x \not\leq y)  \wedge  (y \not\leq x)$ for every element $x$ of $M$ and every element $y$ of $N$. In addition, if $x, x'$ are elements of $M$ and $y$ is an element of $N$ then $M \oplus N  \models (x' \leq x \odot y)$ if and only if $M \oplus N  \models (x' \leq x)$ if and only if  $M \models (x' \leq x)$. Similarly, if $y , y'$ are elements of $N$ then $M \oplus N  \models (y' \leq x \odot y)$ if and only if $M \oplus N  \models (y' \leq y)$ if and only if $N \models (y' \leq y)$. These properties fully characterize $M \oplus N$ so let us use them as ``definition statements'' for the model $M \oplus N$.  \\
We are using the same names for the constants in $F_{C_M \cup C_N}(\emptyset)$, in $M$, in $N$ and in $M \oplus N$. Thanks to this choice we can use the terms of $F_{C_M \cup C_N}(\emptyset)$ as names for elements of $M \oplus N$, of $M$ and also of $N$. Of course, the definition statements above also hold when $x$ and $x'$ are terms of $F_{C_M}(\emptyset)$ and when $y$ and $y'$ are terms of $F_{C_N}(\emptyset)$ (instead of elements of $M$ and $N$). \\
Consider now the model $[A \cup B]$ resulting from the union of the atoms of $M$ and $N$. Since $C_M \cap C_N = \emptyset$ for every term $t$ of $F_{C_M}(\emptyset)$ the lower atomic segments are equal: ${\bf{L}}^a_{[A \cup B]}(t) = {\bf{L}}^a_{M}(t)$. If $t$ and $s$ are terms of $F_{C_M}(\emptyset)$ such that $M \models (t = s)$ then ${\bf{L}}^a_{M}(t) = {\bf{L}}^a_{M}(s)$ which implies ${\bf{L}}^a_{[A \cup B]}(t) = {\bf{L}}^a_{[A \cup B]}(s)$ and then $[A \cup B] \models (t = s)$. It follows that  $M \models (t = s)$ if and only if $[A \cup B] \models (t = s)$, so $M$ can be embedded into $[A \cup B]$ (and so $N$ does). \\
We are going to show that $[A \cup B]$ also satisfies the definition statements. Since the lower atomic segment in $[A \cup B]$ of the terms $x$ and $y$ are non-empty and have null intersection, i.e. ${\bf{L}}^a_{[A \cup B]}(x) \cap  {\bf{L}}^a_{[A \cup B]}(y) = \emptyset$, it follows that $[A \cup B]  \models (x \not\leq y) \wedge (y \not\leq x)$. Since ${\bf{L}}^a_{[A \cup B]}(x \odot y)  = {\bf{L}}^a_{[A \cup B]}(x) \cup  {\bf{L}}^a_{[A \cup B]}(y)$ and ${\bf{L}}^a_{[A \cup B]}(x') \cap  {\bf{L}}^a_{[A \cup B]}(y) = \emptyset $ then $[A \cup B]  \models (x' \leq x \odot y)$ if and only if  $[A \cup B]  \models (x' \leq x)$ if and only if $M \models (x' \leq x)$.  \\
The model atomized by $A \cup B$, i.e. $M + N$, satisfies the definition statements, hence, $M + N$ and $M \oplus N$ are the same model.
\end{proof}

\bigskip

We are now in conditions to define the join model $M \oplus N$ in the case $C_M \cap C_N  \not= \emptyset$.

Let $M$ and $N$ be two semilattices over $C$ such that $C_M \cap C_N  \not= \emptyset$. Rename the constants in the set $C_M \cap C_N$ to obtain disjoint sets $C_M$ and $C'_N$, i.e. $C_M \cap C'_N  = \emptyset$, and model $N' \approx N$. Then, let $M \oplus N \approx (M \oplus N')/\theta$, where $\theta$ is the congruence $\theta = \vee \{ \Theta(c, c') : c \in  C_M \cap C_N \}$.

To be explicit regarding the change of names, a convenient notation is the rename operator $[\tfrac{c'}{c}]$, an operator that replaces the constant $c$ by $c'$. It maps a model with the constant $c$ to the same model with the constant $c$ replaced by $c'$. The rename operator acts over the freest model as, $[\tfrac{c'}{c}]F_{a,..,c}(\emptyset) \mapsto F_{a,..,c'}(\emptyset)$, commutes with all the natural homomorphisms and operates over an atom by replacing $c$ by $c'$ in its upper segment. We also define $[\tfrac{}{c}] \phi$ as an operation that removes $c$ from the upper segment of $\phi$. If the upper segment of $\phi$ becomes empty then $[\tfrac{}{c}]$ annihilates the atom $\phi$.

By using the rename operator we can define the join model as follows. Let $\{c_1,c_2,.., c_h\} = C_M \cap C_N$;
\[
M \oplus N = [\tfrac{c_1}{c'_1}...\tfrac{c_h}{c'_h}] (M \oplus [\tfrac{c'_1}{c_1}...\tfrac{c'_h}{c_h}]N)/\theta
\]

\bigskip

\begin{theorem}  \label{sumOfModelsWithIntersection} 
Let $M$ and $N$ be two atomized semilattices over the constants $C_M$ and $C_N$ respectively, with $C_M \cap C_N \not= \emptyset$.  Rename the constants in the set $\{c_1,c_2,.., c_h\}  = C_M \cap C_N$ to obtain disjoint sets $C_M$ and $C'_N = [\tfrac{c'_1}{c_1}...\tfrac{c'_h}{c_h}]C_N$ and let the renamed model $N' = [\tfrac{c'_1}{c_1}...\tfrac{c'_h}{c_h}]N$. \\ 
i) $M \oplus N$ satisfies $M \oplus N  \approx \square_E (M + N')\,$ where $\,\square_E = \square_{Ec_1}\square_{Ec_2}...\square_{Ec_h}$ and $\,\square_{Ec} = \square_{c  \leq c'} \square_{c'  \leq c}$. \\ 
ii) $M \oplus N  = [|(C_M \cup C_N)] \, \square_E (M + N')$. \\ 
iii) The atoms of $M \oplus N$ are atoms of $M + N$ or redundant with $M + N$.  \\ 
iv) $M + N$ is freer or as free than $M \oplus N$. \\ 
v) There is a homomorphism from $M + N$ onto $M \oplus N$. \\ 
vi) $M \oplus N = \square_{r^{+}_1}\square_{r^{+}_2}....\square_{r^{+}_k} (M + N)$ for some duples $\{r^{+}_1, r^{+}_2,.., r^{+}_k\}$ over $C_M \cup C_N$.
\end{theorem}
\begin{proof}
(i) It follows directly from the definition of the join model and theorem \ref{congruenceFromFullCrossing} to transform congruences into full crossing operations: $M \oplus N  =  [\tfrac{c_1}{c'_1}...\tfrac{c_h}{c'_h}]  \square_E (M + N')$, and then $M \oplus N$ is a rename of $\square_E (M + N')$, so both models are isomorphic. \\ 
(ii) $\square_E$ equates each pair of constants $c_i = c'_i$ for $1 \leq i \leq h$. Since $\,\square_E (M + N') \models \forall_i(c_i = c'_i)$ the atoms in the lower segments of $c_i$ and $c'_i$ are the same, and then every atom with $c_i$ in its upper constant segment also has $c'_i$ and vice versa. Therefore, $\square_E (M + N')$ is the model $[\tfrac{c_1 c'_1}{c_1}...\tfrac{c_h c'_h}{c_h}](M \oplus N)$.  
Let $A = [\tfrac{c_1 c'_1}{c_1}...\tfrac{c_h c'_h}{c_h}](M \oplus N) = \square_E (M + N')$. From theorem \ref{restrictionLemma}, the restriction to $C_M \cup C_N$ of $A$ recovers $M \oplus N$ (as the restriction simply acts by removing the primed constants from the upper segments of the atoms) and we can write this as $A^{|(C_M \cup C_N)} = M \oplus N$ or in the form of an operator, $[|(C_M \cup C_N)] A = M \oplus N$.  \\ 
(iii) Since $A = \square_E (M + N')$ is the result of a crossing operation over $M + N'$ then $M + N'$ is freer than $A$. This can be written as $A \subset M + N'$, where the inclusion signifies that the atoms of $A$ are atoms of $M + N'$ or redundant with $M + N'$. Let $\phi$ be an atom of $A$. We argued above that if $c_i \in C_M \cap C_N$ then either both constants $c_i, c'_i \in U^{c}(\phi)$, or $c_i \not\in U^{c}(\phi)$ and $c'_i \not\in U^{c}(\phi)$. This implies that if $\phi$ is a union of an atom in $M$ and an atom in $N'$ then $\phi^{|(C_M \cup C_N)}$ is a union of an atom in $M$ and an atom in $N$. We have in $A$ two kinds of atoms. Atoms that are unions of atoms of $M$ and $N$ that satisfy $\phi^{|(C_M \cup C_N)} \in M + N$ and atoms that are either in $M$ or in $N$ that contain no constant in the set $C_M \cap C_N$ and satisfy $\phi^{|(C_M \cup C_N)} = \phi \in M + N$. It follows that every atom of $A^{|(C_M \cup C_N)} = M \oplus N$ is either an atom of $M + N$ or a union of atoms of $M + N$. We can also write this as $M \oplus N \subset M + N$.   \\ 
(iv) Proposition iii implies that $M \oplus N$ is a subset model of $M + N$ so, using theorem \ref{construibleSubmodelTheorem}(iv) we get that $M + N$ is freer or as free as $M \oplus N$.   \\ 
(v)  Directly from part (iv) and theorem \ref{homomorphismAndFreedomTheorem}.   \\ 
(vi). From part iv and the fact that $M + N$ and $M \oplus N$ are both models over the same constants $C_M \cup C_N$, follows that $Th_0^{+}(M + N)  \subseteq  Th_0^{+}(M \oplus N)$ so there is some set $R^{+}$ of duples over $C_M \cup C_N$, perhaps empty, such that $Th_0^{+}(M \oplus N) = F_{C_M \cup C_N}(Th_0^{+}(M + N) \cup R^{+})$, for example, the set $R^{+} = Th_0^{+}(M \oplus N)  - Th_0^{+}(M + N)$. Let's enumerate the set $R^{+} = \{r^{+}_1, r^{+}_2,.., r^{+}_k\}$. We can use theorem \ref{fullCrossingIsFreestTheorem} to build the freest model that satisfies $Th_0^{+}(M + N) \cup R^{+}$ as a series of full-crossings $M \oplus N = \square_{r^{+}_1}\square_{r^{+}_2}....\square_{r^{+}_k} (M + N)$.
\end{proof}

\bigskip
\bigskip

We have shown that $M + N$ is freer but not equal to $M \oplus N$. How do they look? 

\bigskip

\begin{example}   Let $M$ be a semilattice model of an algebra with constants $C_M = \{a, b, c\}$ and $N$ a model with constants $C_N = \{c,d, e\}$. Assume that $M = [ \phi_{c}, \phi_{a, b, c} ]$ and $N = [ \phi_{c, d, e} ]$, where we are using the same notation for atoms than in the example \ref{simplecrossingExample}.  \\ 
It easily follows that: \[
M \models (a = b < c)     \, \, \, \, \, \, \, \, \, \, \, N \models (c = d = e). 
\]
The union model $M + N$ is equal $[\phi_{c}, \phi_{a, b, c}, \phi_{c, d, e}]$ and is a model that satisfies: \[
M + N \models (a = b < c > d = e). 
\]
The join model $M \oplus N$ can be obtained by calculating $[|\{a, b, c, d, e\}] \, \square_{c  \leq c'} \square_{c'  \leq c} (M + [\tfrac{c'}{c}] N)$. \\
Step by step: 
$M + [\tfrac{c'}{c}] N = [\phi_{c}, \phi_{a, b, c}, \phi_{c', d, e}]$ and then $\square_{c'  \leq c} [\phi_{c}, \phi_{a, b, c}, \phi_{c', d, e}] = [\phi_{c, c', d, e}, \phi_{a, b, c, c', d, e}]$ and $\square_{c  \leq c'} [\phi_{c, c', d, e}, \phi_{a, b, c, c', d, e}] = [\phi_{c, c', d, e}, \phi_{a, b, c, c', d, e}]$. Finally, the join model is given by the restriction $[\phi_{c, c', d, e}, \phi_{a, b, c, c', d, e}]^{|\{a, b, c, d, e\}}$, which yields $M \oplus N = [\phi_{c, d, e}, \phi_{a, b, c, d, e}]$, a model that satisfies:  \[
M \oplus N  \models (a = b < c = d = e),
\]
and contains embeddings of both $M$ and $N$. It is also clear that $M + N$ is freer than $M \oplus N$ and, in fact, (see theorem \ref{sumOfModelsWithIntersection} part vi) it can be obtained as: $M \oplus N = \square_{c  \leq d} (M + N)$.
\end{example}

\bigskip
\bigskip

In the example above we saw that both $M$ and $N$ could be embedded in $M \oplus N$ but not in $M + N$. The next theorem explains when embeddings in the join model are possible.

\bigskip

\begin{theorem}  \label{embeddingInJoinModelTheorem}  
Let $M$ and $N$ be two semilattice models over the constants $C_M$ and $C_N$ respectively.   \\ 
i) $M \oplus N = F_{C_M \cup C_N}(Th_0^{+}(M) \cup Th_0^{+}(N))$.  \\ 
ii) Either $Th^{+}(M^{|C_N}) \not\subseteq Th^{+}(N)$ or there is an embedding from $N$ into $M \oplus N$ with image $(M \oplus N)^{|C_N} = N$. \\ 
iii) Either there is a duple $r$ over $C_N$ such that $M \models r^{+}$ and $N \models r^{-}$ or there is an embedding from $N$ into $M \oplus N$ with image $(M \oplus N)^{|C_N} = N$. \\ 
iv) If $C_M \cap C_N = \emptyset$ there are embeddings from $M$ and from $N$ into $M \oplus N$.
\end{theorem}
\begin{proof}
(i) From theorem \ref{sumOfModelsWithIntersection}ii we know  $M \oplus N  = [|(C_M \cup C_N)] \, \square_E (M + N')$ where $N' = [\tfrac{c'_1}{c_1}...\tfrac{c'_h}{c_h}]N$. Using theorem \ref{fullCrossingIsFreestTheorem} we can write $\square_E (M + N') = F_{C_M \cup C_N}(Th_0^{+}(M + N')  \cup E) $ with $E = \cup_{\forall i \in C_M \cap C_N}  \{ (c_i  \leq c'_i) \wedge (c'_i  \leq c_i)  \}$. Since the constants of $M$ and $N'$ are disjoint, it is not difficult to see that $Th_0^{+}(M + N') \Leftrightarrow Th_0^{+}(M)  \cup Th_0^{+}(N')$. It follows that $\square_E (M + N') = F_{C_M \cup C_N}(Th_0^{+}(M)  \cup Th_0^{+}(N')  \cup E)$.
 It is also clear that $Th_0^{+}(N')  \cup E \Leftrightarrow Th_0^{+}(N) \cup E$ and then $\square_E (M + N') = F_{C_M \cup C_N}(Th_0^{+}(M)  \cup Th_0^{+}(N)  \cup E)$, where  primed constants $c'_i$ appear only in $E$ and its restriction to $C_M \cup C_N$ is equal to $F_{C_M \cup C_N}(Th_0^{+}(M)  \cup Th_0^{+}(N))$.  \\ 
(ii) We can construct $M \oplus N = F_{C_M \cup C_N}(Th_0^{+}(M) \cup Th_0^{+}(N))$ with a series of full crossings $\square_{Th_0^{+}(N)} F_{C_M \cup C_N}(Th_0^{+}(M))$ where $\square_{Th_0^{+}(N)}$ is a sequence with a full crossing for each duple in $Th_0^{+}(N)$. Consider the restriction: $(M \oplus N)^{|C_N} = [|C_N] \square_{Th_0^{+}(N)}  F_{C_M \cup C_N}(Th_0^{+}(M))$. Since each duple in $Th_0^{+}(N)$ is over the constants $C_N$, restriction and full crossing commute (see theorem \ref{restrictionAndCrossingCommute}) so we can write $(M \oplus N)^{|C_N} = \square_{Th_0^{+}(N)}  [|C_N]  F_{C_M \cup C_N}(Th_0^{+}(M))$. Keeping theorem $\ref{freestModelTheorem}$ in mind it is easy to see that $F_{C_M \cup C_N}(Th_0^{+}(M)) = M +  F_{C_N - C_M}(\emptyset)$. Substituting above, we get $(M \oplus N)^{|C_N} = \square_{Th_0^{+}(N)} \, [|C_N] \, (M +  F_{C_N - C_M}(\emptyset)) =  \square_{Th_0^{+}(N)} \, (M^{|C_N} +  F_{C_N - C_M}(\emptyset))$. Considering that $M^{|C_N} +  F_{C_N - C_M}(\emptyset) = \square_{Th^{+}(M^{|C_N})} F_{C_N}(\emptyset)$ follows $(M \oplus N)^{|C_N} =  \square_{Th_0^{+}(N)} \, \square_{Th^{+}(M^{|C_N})} F_{C_N}(\emptyset)$. \\
Suppose $Th^{+}(M^{|C_N}) \subseteq Th^{+}(N)$. It is clear, then, that   $\square_{Th_0^{+}(N)} \, \square_{Th^{+}(M^{|C_N})} = \square_{Th_0^{+}(N)}$ and $(M \oplus N)^{|C_N} =  \square_{Th_0^{+}(N)} F_{C_N}(\emptyset) = N$. Since $(M \oplus N)^{|C_N}$ is a subalgebra of $M \oplus N$ there is an embedding form $N$ into $M \oplus N$.  \\ 
(iii) $Th^{+}(M^{|C_N}) \not\subseteq Th^{+}(N)$ occurs if and only if there is a duple $r$ over $C_N$ such that $M^{|C_N} \models r^{+}$ and $N \models r^{-}$ and it is a consequence of theorem \ref{restrictionLemma} that $M^{|C_N} \models r^{+}$ if and only if $M \models r^{+}$,  for any duple $r$ over $C_N$. \\ 
(iv) if  $C_M \cap C_N = \emptyset$ there is no duple $r$ satisfying the conditions of proposition (iii), so there should be an embedding from $N$ into $M \oplus N$ (and another embedding from $M$ into $M \oplus N$).
\end{proof} 
\bigskip
\bigskip

Theorem \ref{embeddingInJoinModelTheorem} provides an alternative definition for the join model: \[
M \oplus N = F_{C_M \cup C_N}(Th_0^{+}(M) \cup Th_0^{+}(N)).
\]

We know how to do a restriction $M^{|Q}$, which is a subalgebra generated by a subset $Q \subset C$ of the constants. Consider the more general problem of calculating the subalgebra $S$ of a model $M$ over $C$ generated by a set of $e$ elements represented by terms $\{t_1,t_2,.., t_e\}$ over $C$. Let $G$ be a set with $e$ new constants. We can first extend the model $M$ to a model $M \oplus F_{G}$ over $C \cup G$, and then equate each of the $e$ terms to one constant in $G$ using a principal congruence as follows: \[
S \approx [|G] \, (M \oplus F_{G}(\emptyset))/\theta  \,\,\,\,\,\,\,\,\,\,\,\,\,\,\,\,\,\,  \theta =  \vee_i \Theta(t_i, g_i).
\]
where $G = \{g_1,g_2,.., g_e\}$ are constants. Once we calculate the quotient model $(M \oplus F_{G}(\emptyset))/\theta$ we do a restriction to the subset $G$ of the constants to obtain a model over $G$ that is isomorphic to the subalgebra generated by the terms $\{t_1,t_2,.., t_e\}$.

\bigskip

\begin{theorem}  \label{subalgebraTheorem} 
Let $M$ be an atomized semilattice over $\,C = \{c_1,c_2,.., c_m\}$ and $\{t_1,t_2,.., t_e\}$ a set of terms of $F_{C}(\emptyset)$. Let $G = \{g_1,g_2,.., g_e\}$ be a set of constants. The subalgebra $S$ of $M$ generated by $\{t_1,t_2,.., t_e\}$ and represented using the constants in $G$ is: \\
i) $S = [|G] \Pi^e_{i=1} \square_{g_i \leq t_i}\square_{t_i \leq g_i} (M + F_{G}(\emptyset))$. \\
ii) Let $G^k = \{g^k_1,g^k_2,.., g^k_{e(k)}\} \subseteq G$ such that $(g_i \in G^k) \Leftrightarrow (c_k \in {\bf{C}}(t_i))$. Then $S$ is the renamed model:
$S = [\tfrac{g^1_1 g^1_2 ... g^1_{e(1)}}{c_1} \tfrac{g^2_1 g^2_2 ... g^2_{e(2)}}{c_2} ... \tfrac{g^m_1 g^m_2 ... g^m_{e(m)}}{c_m}] M$.  \\
iii) The number of non-redundant atoms of $S$ is equal or smaller than the number of non-redundant atoms of $M$.
\end{theorem}
\begin{proof}
(i) We argued that $S \approx [|G] \, (M \oplus F_{G}(\emptyset))/\theta$, where the congruence $\,\theta =  \vee_i \Theta(t_i, g_i)$. Since $C \cap G = \emptyset$, theorem \ref{sumOfModelsWhenNoIntersection} says that we can build $M \oplus F_{G}$ simply as $M + F_{G}$, the union of two sets of atoms. Since $\theta$ is a join of congruences we can write, for any model $A$, that $A/\theta = (A/\Theta(t_1, g_1))/\Theta(t_2, g_2).../\Theta(t_e, g_e)$ and theorem \ref{congruenceFromFullCrossing} allows us to build the congruence step by step, in any order, using full crossing operations, as $A/\theta = \Pi^e_{i=1} \square_{g_i \leq t_i} \square_{t_i \leq g_i} A$.  \\
(ii) Theorem \ref{freestModelTheorem} says that $F_{G}(\emptyset)$ can be atomized with $e$ non-redundant atoms $\phi_{g_1}, \phi_{g_2},..., \phi_{g_e}$ each with a single constant $g_i$ in its upper segment. Since $C \cap G = \emptyset$ the crossing $\square_{t_i \leq g_i} (M + F_{G}(\emptyset))$ simply adds $g_i$ to the upper segment of every atom $\phi < t_i$ of $M$. The next crossing $\square_{g_i \leq t_i}$ acts by removing $\phi_{g_i}$ from the resulting atomization. Since $\phi < t_i$ if and only if there is a constant $c_k \in {\bf{C}}(t_i)$ such that $\phi < c_k$ the atom $\phi$ gains $g_i$ if and only if there is a $c_k$ such that $\phi < c_k \in {\bf{C}}(t_i)$ or, in other words if ${\bf{U}}^c(\phi) \cap {\bf{C}}(t_i) \not= \emptyset$. The result is that after computing $\Pi^e_{i=1} \square_{g_i \leq t_i}\square_{t_i \leq g_i} (M + F_{G}(\emptyset))$ all the atoms of $F_{G}(\emptyset)$ are gone and every atom $\phi$ of $M$ has gained $g_i$ if and only if ${\bf{U}}^c(\phi) \cap {\bf{C}}(t_i) \not= \emptyset$. After calculating the full crossings, obtaining the restriction to $G$ can be done by simply removing all the original $c$ constants from the upper segments of the atoms, eliminating every atom that has no constants of $G$ in its upper segment. At the end, each atom of $\phi$ of $M$ has changed to $[\tfrac{g^1_1 g^1_2 ... g^1_{e(1)}}{c_1} \tfrac{g^2_1 g^2_2 ... g^2_{e(2)}}{c_2} ... \tfrac{g^m_1 g^m_2 ... g^m_{e(m)}}{c_m}] \phi$ and the proposition follows from theorem \ref{restrictionLemma}. Note that some $G^k$ may be empty; in this case $\tfrac{\emptyset}{c_k}$ simply removes $c_k$ from the upper segment of the atoms. \\
(iii) Take an atomization for $M$ with just non-redundant atoms and calculate $S$ using proposition ii. The atoms of $S$ are renamed atoms of $M$, hence, $S$ cannot have more non-redundant atoms than $M$.  
\end{proof}

\bigskip
\bigskip

\begin{example}   Consider the subalgebra of $F_{c_1,c_2,c_3}(\emptyset)$ generated by the terms $t_1 = c_1$, $t_2 = c_2$, $t_3 = c_1 \odot c_3$ and $t_4 = c_2 \odot c_3$.  It follows:\[
F_{c_1,c_2,c_3}(\emptyset) \models (t_1 < t_3) \wedge (t_2 < t_4) \wedge (t_2 \odot t_3 = t_4 \odot t_1).
\]
Applying the theorem \ref{subalgebraTheorem}(ii) we can calculate the subalgebra as:\[
S = [\tfrac{g_1 \, g_3}{c_1} \tfrac{g_2 \, g_4}{c_2}\tfrac{g_3 \, g_4}{c_3}] [\phi_{c_1}, \phi_{c_2}, \phi_{c_3}] = [\phi_{g_1 \, g_3}, \phi_{g_2 \, g_4}, \phi_{g_3 \, g_4}], 
\]
a model over $G = \{ g_1, g_2, g_3, g_4 \}$ that satisfies:\[
S \models (g_1 < g_3) \wedge (g_2 < g_4) \wedge (g_2 \odot g_3 = g_4 \odot g_1).
\]
Same result should be obtained from  theorem \ref{subalgebraTheorem}(i): \[
S = [|G] \Pi^e_{i=1} \square_{g_i \leq t_i}\square_{t_i \leq g_i} [\phi_{c_1}, \phi_{c_2}, \phi_{c_3}, \phi_{g_1}, \phi_{g_2}, \phi_{g_3} , \phi_{g_4}] =
\]\[
= [|G]  \square_{g_4 \leq c_2 \odot c_3}\square_{c_2 \odot c_3 \leq g_4}  \square_{g_3 \leq c_1 \odot c_3}\square_{c_1 \odot c_3 \leq g_3}  \square_{g_2 \leq c_2}\square_{c_2 \leq g_2}   \square_{g_1 \leq c_1}\square_{c_1 \leq g_1}      [\phi_{c_1}, \phi_{c_2}, \phi_{c_3}, \phi_{g_1}, \phi_{g_2}, \phi_{g_3}, \phi_{g_4}] =
\]\[
= [|G]  \square_{g_4 \leq c_2 \odot c_3}\square_{c_2 \odot c_3 \leq g_4}  \square_{g_3 \leq c_1 \odot c_3}\square_{c_1 \odot c_3 \leq g_3}   [\phi_{c_1 \, g_1}, \phi_{c_2  \, g_2}, \phi_{c_3}, \phi_{g_3}, \phi_{g_4}] = 
\]\[
= [|G]  \square_{g_4 \leq c_2 \odot c_3}\square_{c_2 \odot c_3 \leq g_4}   [\phi_{c_1 \, g_1  \, g_3}, \phi_{c_2  \, g_2}, \phi_{c_3  \, g_3} , \phi_{g_4}] = 
\]\[
= [|G]   [\phi_{c_1 \, g_1  \, g_3}, \phi_{c_2  \, g_2  \, g_4}, \phi_{c_3  \, g_3  \, g_4}] =   [\phi_{g_1  \, g_3}, \phi_{g_2  \, g_4}, \phi_{g_3  \, g_4}].
\]

\end{example}

\bigskip
\bigskip

It is clear from theorem \ref{subalgebraTheorem} that any renaming of constants produces a subalgebra:

\bigskip

\begin{theorem}  \label{renamesAreSubalgebrasTheorem} 
Let $M$ by an atomized semilattice over $\,C = \{c_1,c_2,.., c_m\}$. Suppose we are given  $m$ sets $G^1, G^2,...G^m$ with $G^k = \{g^k_1,g^k_2,.., g^k_{e(k)}\} \subseteq G$ such that $G \cap C = \emptyset$ and the set $G = \cup_k G^k$ has cardinal $e =  \vert G \vert$. The rename: $[\tfrac{g^1_1 g^1_2 ... g^1_{e(1)}}{c_1} \tfrac{g^2_1 g^2_2 ... g^2_{e(2)}}{c_2} ... \tfrac{g^m_1 g^m_2 ... g^m_{e(m)}}{c_m}] M$ is isomorphic to the subalgebra of $M$ generated by terms $t_1, t_2,...,t_e$ that satisfy: $\,c_k$ is a component constant of $t_i$ if and only if $g_i \in G^k$.
\end{theorem}
\begin{proof}
It follows that $(g_i \in G^k) \Leftrightarrow (c_k \in {\bf{C}}(t_i))$ and we can use theorem \ref{subalgebraTheorem} to state that $[\tfrac{g^1_1 g^1_2 ... g^1_{e(1)}}{c_1} \tfrac{g^2_1 g^2_2 ... g^2_{e(2)}}{c_2} ... \tfrac{g^m_1 g^m_2 ... g^m_{e(m)}}{c_m}] M$ is the subalgebra of $M$ generated by the terms $t_1, t_2,...,t_e$.
\end{proof}

\bigskip
\bigskip

\begin{theorem}  \label{modelAsASubalgebra} 
Let $M$ be an atomized semilattice over $\,C = \{c_1,c_2,.., c_m\}$ and $A = \{\phi_1,\phi_2,.., \phi_z\}$ a set of atoms of cardinal $z = \vert A \vert$ that atomizes $M$. Let $Z$ be a set of $z$ constants. $M$ is the subalgebra of $F_Z(\emptyset)$ generated by the terms $\{t_1,t_2,.., t_m\}$  over $Z$ defined by: $\,z_k$ is a component constant of $t_i$ if and only if $c_i \in {\bf{U}}^c(\phi_k)$.
\end{theorem}
\begin{proof}
$F_Z(\emptyset)$ is atomized by the same number of atoms than $M$, each atom with a single constant $z_i$ in its upper segment. Then $M$ is the rewrite of $F_Z(\emptyset)$ that maps the constants of $F_Z(\emptyset)$ to the upper segments of the atoms of $M$, as follows: $M = [\Pi^z_{k = 1}  \tfrac{{\bf{U}}^c(\phi_k)}{z_k}] F_Z(\emptyset)$. We can now use theorem \ref{renamesAreSubalgebrasTheorem} to claim $M$ as a subalgebra of $F_Z(\emptyset)$ generated by the terms $t_1, t_2,...,t_m$ over $Z$, terms that satisfy $z_k$ is a component constant of $t_i$ if and only if $c_i \in {\bf{U}}^c(\phi_k)$.
\end{proof}

\bigskip
\bigskip

We can define the product of models $M$ and $N$, as usual, as the model  $M \otimes N$ with elements $(g, h)$ where $g$ is an element of $M$ and $h$ and element of $N$. The order relation and the idempotent operator act component-wise, i.e. $(e, f)  \leq (g, h)$ if and only if $e \leq g$ and $f \leq h$. If $M$ is generated by the constants $C_M = \{a_1, a_2,...,a_m \}$ and $N$ is generated by the constants $C_N  = \{b_1, b_2,...,b_n \}$, then $M \otimes N$ has constants $(a_i, b_j)$ and is generated by them.

Suppose that $C_M \cap C_N = \emptyset$. The element $(x, y)$ of $M \otimes N$ can be mapped, one-to-one, to the idempotent summation $x \odot y$ where $x$ is an element of $M$ and $y$ an element of $N$, which works because $e \odot f  \leq g \odot h$ if and only if $e \leq g$ and $f \leq h$. We saw before that an element of the join model $M \oplus N$ looks either like an element $x$ of $M$ or like an element $y$ of 
$N$ or like a summation $x \odot y$. Therefore $M \otimes N$ is isomorphic to a subalgebra of $M \oplus N$, particularly the subalgebra generated by the terms $t_{ij} = a_i \odot b_j$. 

\bigskip

\bigskip
\begin{theorem}  \label{productModelNoIntersectTheorem} 
Let $M$ and $N$ be two semilattice models over the constants $C_M = \{a_1, a_2,...,a_m \}$ and $C_N  = \{b_1, b_2,...,b_n \}$, respectively, such $C_M \cap C_N = \emptyset$. Let $G$ by the set of $m \times n$ constants $g_{ij} = (a_i, b_j)$. The model $M \otimes N$ generated by $G$ satisfies: \\
i) $M \otimes N = [|G] \Pi^{m}_{i = 1} \Pi^{n}_{j=1}  \square_{g_{ij} \leq a_i  \odot b_j} \square_{ a_i  \odot b_j  \leq g_{ij} }  (M + N + F_{G}(\emptyset))$. \\
ii) $M \otimes N = [\tfrac{g_{1,1} g_{1,2} ... g_{1, n}}{a_1} \tfrac{g_{2,1} g_{2,2} ... g_{2, n}}{a_2} ... \tfrac{g_{m,1} g_{m,2} ... g_{m, n}}{a_m}] [\tfrac{g_{1,1} g_{2,1} ... g_{m, 1}}{b_1} \tfrac{g_{1,2} g_{2,2} ... g_{m, 2}}{b_2} ... \tfrac{g_{1,n} g_{2,n} ... g_{m, n}}{b_n}] (M + N)$.
iii) $M \otimes N$ can be atomized with as few non-redundant atoms as the number of non-redundant atoms of $M$ plus the number of non-redundant atoms of $N$.
\end{theorem}
\begin{proof}
(i) Since $M \otimes N$ is isomorphic to the subalgebra of $M \oplus N$ generated by the terms $t_{ij} = a_i  \odot b_j$, we can use theorem \ref{subalgebraTheorem}: $M \otimes N = [|G] \Pi^{m}_{i = 1} \Pi^{n}_{j=1}  \square_{g_{ij} \leq a_i  \odot b_j} \square_{ a_i  \odot b_j  \leq g_{ij} }  ((M \oplus N) + F_{G}(\emptyset))$. Since $C_M \cap C_N = \emptyset\,$ it is possible to substitute $M \oplus N$ by $M + N$. \\
(ii) This is the result of applying theorem \ref{subalgebraTheorem}(ii) to $M \oplus N = M + N$. \\
(iii) Follows directly from proposition (ii), which says that each atom of $M \otimes N$ is a rename of an atom in the set $M + N$. 
\end{proof}
\bigskip
\bigskip

Theorem \ref{productModelNoIntersectTheorem}(ii) gives a very simple way to find an atomization for $M \otimes N$ from atomizations of $M$ and $N$.

\bigskip
 
There is a theorem by Birkhoff that says that every algebraic structure $A$ is isomorphic to a subdirect product of subdirectly irreducible algebraic structures that are homomorphic images of $A$ (see for example  \cite{Burris,Birkhoff}). Stone's representation theorem can be seen as a particular case for Boolean algebras  \cite{Stone}. We give now a constructive proof for the celebrated result that any non-trivial semilattice is a subdirect product of two-element semilattices. In the next theorem we explicitly build the subdirectly irreducible components and show that each component maps to an atom of $M$ and is a semilattice with two elements and spawned by two atoms.

\bigskip
\bigskip

\begin{theorem}  \label{modelAsASubdirectProduct} 
Let $M$ be an atomized semilattice over $\,C = \{c_1,c_2,.., c_m\}$ and $A = \{\phi_1,\phi_2,.., \phi_z\}$ a non-empty set of atoms of cardinal $z = \vert A \vert$ such that all $\phi_i \not= \ominus_c$ and $M = [ A \cup  \ominus_c]$. Assume $M$ is not trivial, i.e. $M \not= [\ominus_c]$. Let $B_1, B_2,...,B_z$ a set of $z$ models, such that $B_j = [\psi_{z_j}, \psi_{z_j \overline{z}_j}]$ is a model over two constants $z_j, \overline{z}_j$ with $U^c(\psi_{z_j}) = \{z_j\}$ and $U^c(\psi_{z_j \overline{z}_j}) = \{z_j, \overline{z}_j\}$ that satisfies $B_j \models (\overline{z}_j \lneq z_j)$ and is equal to the homomorphic image of $M$ under the homomorphism: $h_j:t  \mapsto z_j$ if $\phi_j < t$ and  $h_j:t  \mapsto \overline{z}_j$ if $\phi_j \not< t$. \\
 i) $M$ is isomorphic to the subalgebra of $\otimes^z_{j=1} B_j$ generated by the constants $u_{c_1},u_{c_2},..., u_{c_m}$, each equal to a tuple $u_{c_i} = (u_{i1}, u_{i2},..., u_{iz})$ with $u_{ij} = z_j$ if $\phi_j < c_i$ or $u_{ij} = \overline{z}_j$ if $\phi_j \not< c_i$.  \\ 
ii) $M$ is isomorphic to the subalgebra of $\oplus^z_{j=1} B_j$ generated by the terms $\{t_1,t_2,.., t_m\}$ over $2Z = \cup^{z}_{j =1} \{ z_j, \overline{z}_j \}$ defined by: $\,z_k$ is a component constant of $t_i$ if and only if $\phi_k < c_i$ and $\,\overline{z}_k$ is a component constant of $t_i$ if and only if $\phi_k \not< c_i$. 
\end{theorem}
\begin{proof}
(i) Consider an embedding $h$ that maps each constant $c_i$ of $M$ to a tuple $u_{c_i}$ with $z$ components, that has component $j$ equal to $z_j$ when $\phi_j < c_i$, and equal to $\overline{z}_j$ when $\phi_j \not< c_i$. Note that $\otimes^z_{j=1} B_j$ has $2^z$ constants and that, since $\odot$ operates component-wise in a product and $\overline{z}_j \lneq z_j$, every summation of constants of $\otimes^z_{j=1} B_j$ produce a constant of $\otimes^z_{j=1} B_j$. Our embedding $h$ maps each regular element $x$ of $M$ to a z-tuple $u_x$ that is a constant of $\otimes^z_{j=1} B_j$ and has, at position $j$: either $z_j$ if $\phi_j < x$ or $\overline{z}_{j}$ if $\phi_j \not< x$, and then $\otimes^z_{j=1} B_j \models (u_x \leq u_y)$ if and only if $M \models (x \leq y)$. Therefore, $M$ is isomorphic to the subalgebra of $\otimes^z_{j=1} B_j$ spawned by the constants $u_{c_1},u_{c_2},..., u_{c_m}$, as we wanted to proof.    \\
Let us check that it is possible to use theorem \ref{productModelNoIntersectTheorem} to build $M$ as a subalgebra of $\otimes^z_{j=1} B_j$, i.e. as a rename of $B_1 + B_2 +....+ B_z$. The union $B_1 + B_2 +....+ B_z$ is atomized by $\cup^z_{j = 1} \{\psi_{z_j}, \psi_{z_j \overline{z}_j}\}$, i.e. it is a model with $2z$ non-redundant atoms. Let $H$ be the set of $2^z$ constants of $\otimes^z_{j=1} B_j$. Each $h \in H$ is a z-tuple. The product is then equal to \[
\otimes^z_{j=1} B_j = [\Pi^z_{j = 1} \tfrac{ h \in H \text{ : } h_j = z_j }{z_j} \tfrac{ h \in H \text{ : } h_j = \overline{z}_j }{\overline{z}_j}]  [ \cup^z_{j = 1} \{\psi_{z_j}, \psi_{z_j \overline{z}_j}\}],
\]
and the subalgebra of $\otimes^z_{j=1} B_j$ generated by elements $u_{c_1},u_{c_2},..., u_{c_m}$ named as $c_1,c_2,.., c_m$ can then be obtained as:\[
M = [ \Pi^{2^z}_{k = 1}  \tfrac{c_i \,\in\, C \text{ : } h_k = u_{c_i} }{h_k} ]    \otimes^z_{j=1} B_j,
\]
a rename that annihilates most constants of $ \otimes^z_{j=1} B_j$. In fact, among the $2^z$ constants, the rename leaves just the $m$ constants $u_{c_1},u_{c_2},..., u_{c_m}$. Putting both expressions together: \[
M = [\Pi^z_{j = 1} \tfrac{ c_i \,\in\, C \text{ : }  u_{ij} = z_j }{z_j} \,\, \tfrac{ c_i \,\in\, C \text{ : }  u_{ij} =\overline{z}_j}{\overline{z}_j}]  [ \cup^z_{j = 1} \{\psi_{z_j}, \psi_{z_j \overline{z}_j}\}],
\]
which renames $\psi_{z_j}$ to $\phi_j$ and renames every $\psi_{z_j \overline{z}_j}$ to $\ominus_c$.   \\
Note that, without loss of generality (except for excluding the trivial semilattice), we had to require the atoms of $A$ to satisfy $\phi_i \not= \ominus_c$ because, if not, the homomorphism from $M$ to some component $B_i$ may not be surjective.  \\
(ii) This proposition follows from proposition (i) which says that $M$ is a subalgebra of $\otimes^z_{j=1} B_j$, and from the fact that $\otimes^z_{j=1} B_j$ is a subalgebra of $\oplus^z_{j=1} B_j$. To be explicit, we are going to give a couple of constructions of $M$ as a subalgebra of $\oplus^z_{j=1} B_j$. \\
Consider a construction very similar to the one made for proposition (i) but with z-tuples replaced by terms of $z$ components. Each constant $c_i$ of $M$ maps to a term $t_i$ over $2Z$ with component constant $z_j$ if $\phi_j < c_i$, or with component constant $\overline{z}_j$ if $\phi_j \not< c_i$. Since $\odot$ is associative and commutative and taking into account that $\overline{z}_j \lneq z_j$ then $\odot$ operates over the terms $t_i$ producing elements that can be represented with a term with $z$ component constants. For each regular element $x$ of $M$, the idempotent summation, $\odot_{i:\, c_i \leq x} t_i$ that runs along every $i$ such that $c_i \leq x$ produces a term $t_x$ over $2Z$ with exactly $z$  component constants, a term that has either component constant $z_j$ if $\phi_j < x$ or component constant $\overline{z_j}$ if $\phi_j \not< x$. It follows that $\oplus^z_{j=1} B_j \models (t_x \leq t_y)$ if and only if $M \models (x \leq y)$, so $M$ is isomorphic to the subalgebra generated by the terms $\{t_1,t_2,.., t_m\}$ which proves the theorem.   \\
Here is another explicit construction: Since the constants of the different $B_j$ are all pairwise disjoint $\oplus^z_{j=1} B_j  =  B_1 + B_2 +....+ B_z = [\cup^z_{j = 1} \{\psi_{z_j}, \psi_{z_j \overline{z}_j}\}]$ and $M$ is the rename: \[
M = [\Pi^z_{j = 1} \tfrac{ c_i \,\in\, C \text{ : }   \phi_j < c_i  }{z_j} \,\, \tfrac{  c_i\,\in\, C \text{ : }   \phi_j \not< c_i  }{\overline{z}_j} ]  [ \cup^z_{j = 1} \{\psi_{z_j}, \psi_{z_j \overline{z}_j}\}],
\]
that renames $\psi_{z_j}$ to $\phi_j$ and  $\psi_{z_j \overline{z}_j}$ to $\ominus_c$.  Since $M$ is a rename of the constants of $\oplus^z_{j=1} B_j$, theorem \ref{renamesAreSubalgebrasTheorem} tells us that $M$ is a subalgebra of $\oplus^z_{j=1} B_j$. \\
Equivalently, we could choose the rename: \[
M = [\Pi^z_{j = 1} \tfrac{ c_i\,\in\, C \text{ : }  \phi_j < c_i  }{z_j} \, \tfrac{ C }{\overline{z}_j} ]  [ \cup^z_{j = 1} \{\psi_{z_j}, \psi_{z_j \overline{z}_j}\}],
\]
that also renames $\psi_{z_j}$ to $\phi_j$ and  $\psi_{z_j \overline{z}_j}$ to $\ominus_c$. The difference here is that the element $x$ is mapped to a term $t'_x$ over $2Z$, a term that has the component constant $z_j$ if $\phi_j < x$ and always has the component constant $\overline{z_j}$.  Note that  $\oplus^z_{j=1} B_j \models \forall x (t_x = t'_x)$.
\end{proof}

\bigskip
\bigskip

We have seen in theorem \ref{modelAsASubalgebra} that $M$ is a subalgebra of $F_Z(\emptyset)$, a model with $z$ constants with $z$ the number of atoms of $M$. We have seen in theorem \ref{modelAsASubdirectProduct} that $M$ is a subalgebra of the product $\otimes^z_{j=1} B_j$, a model with $2^{z}$ constants and also that $M$ is a subalgebra of $\oplus^z_{j=1} B_j$, a model with $2z$ constants.

\bigskip
\bigskip

Let  $C_M = \{a_1, a_2,...,a_m \}$ and $C_N  = \{b_1, b_2,...,b_n \}$. We have seen so far that, when $C_M \cap C_N = \emptyset$ the join model $M \oplus N$ is equal to the union $M + N$ and the product $M \otimes N$ is the rename of $M +N$: \[
 M \otimes N = [\Pi^{m}_{i = 1} \tfrac{g_{i,1} g_{i,2} ... g_{i, n}}{a_i}] [\Pi^{n}_{j=1}  \tfrac{g_{1,j} g_{2,j} ... g_{m, j}}{b_j}] (M + N) = 
\]\[
= [\Pi^{m}_{i = 1} \tfrac{g_{i,1} g_{i,2} ... g_{i, n}}{a_i}] M +  [\Pi^{n}_{j=1}  \tfrac{g_{1,j} g_{2,j} ... g_{m, j}}{b_j}] N,
\]
which corresponds to the subalgebra of $M \oplus N$ generated by the terms $t_{ij} = a_i  \odot b_j$ named as constants $g_{ij}$.

We gave in theorem \ref{sumOfModelsWithIntersection} a description of the join model $M \oplus N$ for the general case when $C_M \cap C_N  \not= \emptyset$ with no ambiguity. The definition of $M \otimes N$ in the general scenario, i.e. with constants shared among $M$ and $N$ is not that unambiguous. Suppose we define $M \otimes N$ as the subalgebra of $M \oplus N$ generated by the terms $t_{ij} = a_i  \odot b_j$. If we choose this definition and $c_u, c_v \in C_M \cap C_N$, the constants $(c_u, c_v)$ and $(c_v, c_u)$ become equal. This assumption is not necessarily true in a product of models. The alternative choice is treating models $M$ and $N$ as if their constants were disjoint:\[
M \otimes N \approx W =  [\Pi^{m}_{i = 1} \tfrac{g_{i,1} g_{i,2} ... g_{i, n}}{a_i}]M +  [\Pi^{n}_{j=1}  \tfrac{g_{1,j} g_{2,j} ... g_{m, j}}{b_j}] N.
\]
Optionally, we can identify, for every element $c \in C_M \cap C_N$ the pair $(c, c)$ with $c$. This can be done by renaming the common constants $g_{i,i}$ to $c_i$, i.e with $
M \otimes N = [\Pi_{i\, :\, c_i \in C_M \cap C_N} \tfrac{c_i}{g_{i,i}}] W.$

\section{Acknowledgments} 

We thank Nabil Abderrahaman Elena and Antonio Ricciardo for critical comments on the manuscript. We are grateful for the support from Champalimaud Foundation (Lisbon, Portugal) provided through the Algebraic Machine Learning project, from Portuguese national funds through FCT in the context of the project UIDB/04443/2020, and from the European Commission provided through projects H2020 ICT48 \emph{Humane AI; Toward AI Systems That Augment and Empower Humans by Understanding Us, our Society and the World Around Us} (grant $\# 820437$) and the H2020 ICT48 project \emph{ALMA: Human Centric Algebraic Machine Learning} (grant $\# 952091$). 

\newpage

\section{Notation and basic definitions} \label{Notation}

The reader should be familiar with the following notation and definitions:

We use a minus symbol for the subtraction of sets, i.e. we use $A - B$ instead of $A  \setminus B$. 

We use capital letters such as $M$ or $N$ to refer to semilattice models or to refer to atomized semilattice models. We use small caps to refer to elements of a model. We sometimes use the same letter for a model and the set of atoms of the model. We use Greek letters to refer to atoms.  

A term is a recipe to form an element by using the constants and the idempotent operator, for example $t = c_1 \odot  c_2 \odot c_3$. Multiple terms can yield the same element in a semilattice model. Two terms $s$ and $t$ yield the same element in a model $M$ if $M \models (s \leq t) \wedge (t \leq s)$. 

Atomized semilattices have elements of two sorts, the regular elements and the atoms. Regular elements form a semilattice. Regular elements and atoms form a partial order. 

The free model (or freest model) $F_C(\emptyset)$ over $C$ is the model with a different element for each term over $C$. We sometimes refer to elements of $F_C(\emptyset)$ with the word ``term''. Terms are either constants or idempotent summations of constants formed using the idempotent, binary operator $\odot$. When a term is equal to a single constant we usually refer to it with the word `constant''. Two terms are equal in the freest model if and only if they can be proven equal by using the commutative, associative and idempotent properties, so $c_1 \odot c_2 \odot c_1$ is the same term than $c_2 \odot c_1$. 

Each regular element of an atomized semilattice model $M$ over a set of constants $C$ corresponds to an equivalence class in $F_C(\emptyset)$.  Each atom element of $M$ is defined by a subset of constants taken from $C$.

We use bold to represent functions that yield sets.  We use ${\bf{C}}(M)$ for the constants and and ${\bf{A}}(M)$ for the atoms of a model $M$. The lower and upper segments of an element $x$ are defined, asymmetrically, as ${\bf{L}}_M(x) = \{y:y < x\,\vee \,y = x\}$ and ${\bf{U}}_M(x) = \{y: y > x\}$. The superscript ${ \ }^{a}$ is used to denote the intersection with the atoms: ${{\bf{L}}^{a}_M}(x) = {\bf{L}}_M(x) \cap {\bf{A}}(M)$. We also use the superscript ${ \ }^{c}$ to represent the intersection with the constants ${{\bf{L}}^{c}_M}(x) = {\bf{L}_M}(x) \cap {\bf{C}}(M)$ and  ${{\bf{U}}^{c}_M}(x) = {\bf{U}_M}(x) \cap {\bf{C}}(M)$.  

In an atomized semilattice model, ${\bf{U}}^{c}_M(\phi)$ is defined as the set of constants in the upper segment of the atom $\phi$. Since an atom is defined by the constants in its upper segment we often drop the subindex $M$ and simply write ${\bf{U}}^{c}(\phi)$. Atoms can be identified across models by its upper constant segment.  

We say an atom $\phi$ is \emph{wider} than atom $\eta$ if it is different than $\eta$ and for every constant $c$, $(\eta < c) \Rightarrow (\phi < c)$. Equivalently an atom $\phi$ is \emph{wider} than atom $\eta$ when ${{\bf{U}}^{c}}(\eta) \subsetneq {{\bf{U}}^{c}}(\phi)$.

An atom $\phi$ is \emph{redundant} with model $M$ iff for each constant $c$ such that $\phi < c$ there is at least one atom $\eta < c$ in $M$ such that $\phi$ is wider than $\eta$. An atom of $M$ that is not redundant with $M$ is called a ``non-redundant'' atom of $M$. 

We say an atom $\phi$ is ``in  $M$'' or is ``an atom of $M$'' and write $\phi \in M$ if, there is a set of atoms $A$ such that $M = [A \cup \{ \phi \}]$.

An atom is ``external'' to $M$, written $\phi \not\in M$, if it is not in $M$. With respect to a model $M$ an atom can be either a `non-redundant atom in $M$ or redundant with $M$ or external to $M$.  

The component constants ${\bf{C}}(a)$ of a constant or term is a set defined as ${\bf{C}}(c)  = \{c\}$ if $c$ is a constant and if $t$ is a term $t = c_1 \odot  c_2 \odot ... c_{n-1} \odot c_n$, as the set ${\bf{C}}(t)  = \{c_1,c_2,...,c_n\}$. Atoms have no component constants. 

We use the word "duple" to refer to an ordered pair of elements $r \equiv (a, b)$ of $F_C(\emptyset)$. We say a model $M$ satisfies the positive duple $r^{+}$ if $M \models (a \leq b)$ and satisfies the negative duple $r^{-}$ if $M \not\models (a \leq b)$ or, equivalently, $M \models (a \not\leq b)$.

A set of positive duples $R^{+}$ of a model $M$ is a set of duples that are satisfied by M. When we write $M \models R^{+}$ we mean $M \models r^{+}$ for all $r \in R^{+}$. In the same way, when we write $M \models R^{-}$ we mean $M \models r^{-}$ for all $r \in R^{-}$. 

The theory of a model $M$, written $Th(M)$, is the set of all first order sentences that are true in $M$, and $Th_0(M)$ is used to refer to atomic sentences without quantifiers that are true in $M$. We use here $Th_0(M)$ as the set of positive and negative duples satisfied by the $M$. The positive and negative theories of $M$, written $Th_0^{+}(M)$ and $Th_0^{-}(M)$ are, respectively, the sets of duples satisfied by $M$ (positive duples of $M$) and the set of duples not satisfied by $M$ (negative duples of $M$).

We say element $y$ contains $x$ if $x \leq y$. 

The atom $\ominus_c$ is the atom defined by ${\bf{U}}^{c}(\ominus_c) = C$.

When $A$ is a set of atoms we often refer to the model spawned by the atoms in $A$ simple as $A$, and sometimes by using brackets $[A]$.

The discriminant ${\bf{dis}}_M(a,b)$ in an atomized model $M$ is the set ${{\bf{L}}^{a}_M}(a) - {{\bf{L}}^{a}_M}(b)$ of atoms of $M$. The duple $a \leq b$ holds in $M$ if and only if ${\bf{dis}}_M(a,b)$ is empty. 

For each atom $\phi \not= \ominus_c$, the ``pinning term'' $T_{\phi}$ is equal to the idempotent summation of all the constants in the set $C - {\bf{U}}^{c}(\phi)$. For each constant $c \in {\bf{U}}^{c}(\phi)$ define the ``pinning duple'' as the negative atomic sentence $(c \not< T_{\phi})$.  $PR(\phi)$ is defined as the set of pinning duples of an atom $\phi$. 

We often used $R$ to refer to a set of positive and/or negative duples, i.e. a set of duples each with its own sign (signed duples). We use $R^{+}$ and $R^{-}$ for the positive and the negative duples in the set so $R = R^{+} \cup R^{-}$.

A model $A$ over the constants $C$ is ``freer or as free" then model $B$ over the constants $K \subseteq C$ if for every duple $r$ over $K$, $B \models r^{-} \Rightarrow A \models r^{-}$. In other words, 
$Th_0^{-}(B)  \subseteq Th_0^{-}(A)$. We sometimes use the word freer for short to mean freer or as free.

The freest model $F_{C}({R^{+}})$ is the model such that if $F_{C}({R^{+}}) \models r^{+}$ for some duple $r$ then every model of $R^{+}$ also satisfies $r^{+}$. When every model of $R^{+}$ satisfies $r^{+}$ we write $R^{+} \Rightarrow r^{+}$ or $R^{+} \vdash r^{+}$.

We use $A + B$ or $[A \cup B]$ to refer to the model spawned by the atoms of $A$ and the atoms of $B$. 

$\bigtriangledown$ is the ``union'' operator of atoms. The union of atoms is an operation that produces an atom $\phi \bigtriangledown \psi$ such that ${\bf{U}}^c(\phi \bigtriangledown \psi) \equiv {\bf{U}}^c(\phi) \cup {\bf{U}}^c(\psi)$. 

The ``full crossing'', $\square_r M$, of a duple $r = (r_L,r_R)$ in an atomized model $M$ such that $M \models r^{-}$ is a mechanism that creates another model $N \models Th_0^{+}(M) \cup \{r^{+}\}$. To calculate the full crossing, define the sets $H = {\bf{dis}}_M(r)$ and $B = {\bf{L}}^a_M(r_R)$ and obtain the model atomized by the set $N = (M - H) \cup (H\bigtriangledown B)$ where the letter $M$ has been used to represent the set of atoms of model $M$ and where $H\bigtriangledown B \equiv \{\lambda \bigtriangledown \rho : (\lambda \in H) \wedge (\rho \in B) \}$. 

For models $M$ and $N$ we write $M \subseteq N$ if for every atom $\phi \in M$ also holds $\phi \in N$. 

\bibliographystyle{unsrt}
\bibliography{main}
\end{document}